\documentclass[10pt]{article}

\usepackage[margin=1.05in]{geometry}
\usepackage{amsmath,amssymb,amsthm,mathtools}
\usepackage{authblk}
\usepackage{enumitem}
\usepackage{booktabs}
\usepackage{longtable}
\usepackage{array}
\usepackage{bbm}
\usepackage{dsfont}
\usepackage{hyperref}
\hypersetup{colorlinks=true,linkcolor=blue,citecolor=blue,urlcolor=blue}
\setlist{topsep=0.35em,itemsep=0.15em,parsep=0pt}

\newtheorem{theorem}{Theorem}[section]
\newtheorem{conjecture}{Conjecture}[section]
\newtheorem{lemma}[theorem]{Lemma}

\newtheorem{claim}[theorem]{Claim}

\newtheorem{assumption}[theorem]{Assumption}

\theoremstyle{definition}

\newtheorem{conclusion}[theorem]{Conclusion}

\newcommand{\ONE}{\mathrm{ONE}}

\newcommand{\F}{\mathcal F}
\newcommand{\M}{\mathcal M}

\title{{\Large\bf The Erd\H{o}s Matching Conjecture for 4-uniform hypergraphs}}
\author[1]{Jianfeng Hou\thanks{Email: \texttt{jfhou@fzu.edu.cn}}}
\author[1]{Caiyun Hu\thanks{Email: \texttt{hucaiyun.fzu@gmail.com}}}
\author[2]{Xizhi Liu\thanks{Email: \texttt{liuxizhi@ustc.edu.cn}}}
\affil[1]{\small Center for Discrete Mathematics, Fuzhou University, Fuzhou, China}
\affil[2]{\small School of Mathematical Sciences, University of Science and Technology of China, Hefei, Anhui, China}
\date{\today}

\begin{document}
\maketitle

\begin{abstract}
We prove the Erd\H{o}s Matching Conjecture for $4$-uniform hypergraphs for every matching number 
$s\ge6004$.  Specifically, for integers $s\ge 6004$ and  $n\ge 4s+4$, 
if a $n$-vertex 4-uniform hypergraph $\mathcal F$  does not contain a matching of size $s+1$, then 
\[
 |\mathcal F|\le
 \max\left\{\binom{4s+3}{4},\binom n4-\binom{n-s}{4}\right\}.
\]

Our proof introduces a finite-board reduction for general uniformity.  It
reduces the global conjecture to a lower-uniformity bound and a fixed finite
optimization at the two adjacent vertex numbers where the candidate
constructions exchange dominance.  In the $4$-uniform case the board has $19$
vertices, and its weighted inequality splits into $19$-, $15$-, and
$11$-vertex layers.  The hardest layer is resolved by exact rational dual
certificates and deterministic integer searches.  All computer-assisted steps
are checked in exact arithmetic by verifiers that reconstruct the finite
systems directly from their mathematical definitions.
\end{abstract}

\section{Introduction}

Given an integer $r\ge 2$, an \emph{$r$-uniform hypergraph} (henceforth an \emph{$r$-graph}) $\mathcal{H}$ is a collection of $r$-subsets of some set $V$. 
We call $V$ the \emph{vertex set} of $\mathcal{H}$ and denote it by $V(\mathcal{H})$. 
When $V$ is understood, we usually identify a hypergraph $\mathcal{H}$ with its edge set and use $|\mathcal{H}|$ to denote the number of edges of $\mathcal{H}$. Let $\nu(\F)$ denote the  \emph{matching number} of $\F$, i.e., the size of a maximum  matching of $\F$.

A classical problem, due to Erd\H{o}s \cite{Erdos1965}, asks for the maximum number of edges in an $r$-graph with a given matching number. 
Define
\[
m_r(n,s):=\max\left\{|\mathcal F|:\mathcal F\subseteq \binom{[n]}{r},\ \nu(\mathcal F)\le s\right\}.
\]

Note that there are two natural extremal constructions for $m_r(n,s)$.  Let 
\[
\mathcal{A}_r(n,s)=\mathcal{K}_{rs+r-1}^r\cup (n-(rs+r-1))\mathcal{K}_1^r~~\text{and}~~\mathcal{B}_r(n,s)=\mathcal{K}_{n}^r-\mathcal{K}_{n-s}^r,
\]
where $\mathcal{K}_{s}^r$ denote the complete $r$-graph on $s$ vertices. Then 
\[
|\mathcal{A}_r(n,s)|= a_r(s)\coloneqq\binom{rs+r-1}{r}, \text{ and }
|\mathcal{B}_r(n,s)|= b_r(n,s)\coloneqq\binom nr-\binom{n-s}{r}.
\]
 Moreover, both $\mathcal{A}_r(n,s)$ and $\mathcal{B}_r(n,s)$ have matching number $s$.

The following conjecture, known as the Erd\H{o}s Matching Conjecture (EMC for short), is one of the central extremal problems and recovers the Erd\H{o}s--Ko--Rado theorem \cite{EKR1961} when $s=1$.  
\begin{conjecture}[EMC, Erd\H{o}s \cite{Erdos1965}]\label{Erdos-Matching-Conjecture}
For all integers \(r\ge2\), \(s\ge0\), and \(n\ge r(s+1)\),
\begin{equation}\label{eq:EMC-general}
 m_r(n,s)
 =
 \max\left\{
a_r(s),
 b_r(n,s)
 \right\}.
\end{equation}
\end{conjecture}

For \(r=2\), Conjecture \ref{Erdos-Matching-Conjecture} follows from the classical  Erd\H{o}s-Gallai theorem 
\cite{ErdosGallai1959}.  For $r\ge 3$, much of the history of this conjecture can be organized
around how close \(n\) is to the lower admissible boundary \(r(s+1)\). 
Erd\H{o}s \cite{Erdos1965} already proved the conjecture for fixed $r$ and $s$ when $n \ge n_0(r,s)$. 
Subsequent work made the bounds explicit: Bollob\'as, Daykin, and Erd\H{o}s \cite{BDE1976} established 
the linear range $n \ge 2r^3 s$, later improved by Huang, Loh, and Sudakov \cite{HLS2012} to $n \ge 3r^2 s$.
Frankl, \L uczak, and Mieczkowska \cite{FLM2012} established the cover-side bound, with uniqueness, for $n>\frac{2r^2s}{\log r}$.

A major advance came from Frankl \cite{Frankl2013}, who showed $ m_r(n,s)=b_r(n,s)$ for $n \ge (2r-1)s + r$. For fixed $r$ and sufficiently large $s$, Frankl and
Kupavskii~\cite{FranklKupavskii2022} reduced this range to $ n\ge (5r-2)s/3$. 
After the first version of the present paper appeared, Cao, Liu, and Zhang
\cite{CaoLiuZhang2026} further improved the general cover-side range to
$n\ge(r+1)s$.  
At the opposite, clique-dominant end, the boundary case
\(n=r(s+1)\) follows from a classical theorem of Kleitman
\cite{K1968}. 
Frankl \cite{Frankl2017NewRange} proved that,
for every fixed \(r\), there exists \(\varepsilon_r>0\) such that
the EMC holds throughout $r(s+1)\le n<(r+\varepsilon_r)(s+1)$.
Stability and related refinements were developed in \cite{FranklKupavskii2017, GLM2022}; related matching and degree versions include~\cite{AlonEtAl2012,Han2016,HuangZhao2017}.
For related results on rainbow matchings in uniform hypergraphs, see \cite{GLMY2022,LWY2022,LWY2023,LYY2021}.

Before the present work, the only nontrivial uniformities for which the full conjecture is known for every admissible \(n\) are \(r=2\) and \(r=3\). 
In the three-uniform case, Frankl, R\"odl, and Ruci\'nski \cite{FRR2012} first proved the conjecture in the range \(n\ge4(s+1)\), in the notation of the present paper.  
\L uczak and Mieczkowska \cite{LuczakM2014} subsequently established the full
admissible range for all sufficiently large \(s\), and Frankl \cite{Frankl2017} completed the proof for every \(s\).
Frankl, R\"odl, and Ruci\'nski \cite{FRR2017} later gave a short proof for
\(s\ge33\); this is the form used as the lower-uniformity input in Theorem~\ref{lem:FRR2017}.
For \(r=4\), a recent work of Frankl, Lu, Ma, and Wu
\cite{FranklLuMaWu2026} proves stability results and establishes the
conjecture  when \(n\ge5s\), with \(n\) sufficiently large. The result of Cao, Liu, and Zhang
\cite{CaoLiuZhang2026} also gives this range when $r=4$. 
The present paper addresses the remaining all-admissible-\(n\) problem for
\(4\)-graphs when \(s\ge6004\).

\begin{theorem}\label{thm:main}
For integers $s\ge6004$ and $n\ge4s+4$, every $n$-vertex $4$-graph
$\F$ with $\nu(\F)\le s$ satisfies
\[
|\F|\le \max\left\{\binom{4s+3}4,\binom n4-\binom{n-s}4\right\}.
\]
\end{theorem}

For  $r\ge 2$, define the \emph{transition value}
\begin{equation*}\label{eq:nr-def}
 n_r(s)\coloneqq\min\left\{n\ge r(s+1):\binom{rs+r-1}{r}\le \binom nr-\binom{n-s}{r}\right\}.
\end{equation*}

Since $a_r(s)$ is independent of $n$, whereas $b_r(n,s)$ is strictly
increasing in $n$, the consecutive values $n_r(s)-1$ and $n_r(s)$ straddle
the transition between the two extremal constructions.  A central
feature of the proof is that it suffices to establish the required bound at 
these two values of $n$.

The proof is organized in two stages.  First, inspired by the work of Frankl, R\"odl, and Ruci\'nski \cite{FRR2017}, we 
prove the general finite-board criterion, 
Theorem~\ref{thm:finite-board-criterion}, in Section~\ref{sec:general-reduction}.  After shifting, the critical-value reduction, maximalization,
and the property-$\ONE$ branch---in which a stable family has a maximum matching
avoiding vertex $1$---is reduced to two inputs:
\begin{enumerate}[label=(\roman*)]
\item the $(r-1)$-uniform EMC gives the cover-side bound
$\binom m{r-1}-\binom{m-t}{r-1}$ for every $m$-vertex link with matching number
at most $t$, for all $m\ge n_r(t)$; and
\item a finite optimization problem for mixed-size trace configurations on an
$(r^2+r-1)$-vertex board, arising from the two critical parameter pairs
$n\in\{n_r(s)-1,n_r(s)\}$.
\end{enumerate}



The non-$\ONE$ deletion chain then transfers the result to arbitrary
families, with an explicit loss in the lower threshold for $s$.

Second, we verify the two inputs for $r=4$.  The known $3$-uniform theorem
provides the link bound, while
Sections~\ref{sec:board-setup}--\ref{sec:r4-board-input-proof} prove the finite
input on a $19$-vertex board.  The board comparison decomposes into three
weighted local inequalities on the $19$-, $15$-, and $11$-vertex layers.  The
hard part of the wide layer splits into top-star and no-top-star branches.
The former is verified by exact rational Farkas certificates; the latter is
established by deterministic finite up-set searches, exact arithmetic, and
rectangle-extension enumeration.

The local board input is verified uniformly for every $s\ge3481$ at
$n\in\{n_4(s)-1,n_4(s)\}$, and hence settles the property-$\ONE$ case in that
range.  The quantitative non-$\ONE$ deletion argument then yields the final
threshold $s\ge6004$ in Theorem~\ref{thm:main}.  We have made no attempt to
optimize this constant.

Appendix~\ref{app:certificate-specs} specifies the
ancillary scripts, their proof-critical success markers, and the mathematical
assertion verified by each script.  All calculations and certificates are
available at
\href{https://github.com/xliu2022/xliu2022.github.io/blob/main/EMC_certificate.zip}{\texttt{EMC\_certificate.zip}}.

\section{A finite-board criterion for general \texorpdfstring{$r$}{r}}\label{sec:general-reduction}

In this section, we establish a general finite-board criterion for Conjecture \ref{Erdos-Matching-Conjecture}.

\subsection{Critical values and the two inputs}\label{Critical-values}

For fixed $r\ge 2$, recall that
\begin{equation}\label{eq:nr-def}
 n_r(s)\coloneqq\min\left\{n\ge r(s+1):\binom{rs+r-1}{r}\le \binom nr-\binom{n-s}{r}\right\}.
\end{equation}
For  a real number $x\geq 1$, define the function 
\[
f_r(x)=x^r-(x-1)^r-r^r.
\]
It is clear that $f_r(x)$ is strictly increasing on $(1,\infty)$, satisfies $f_r(r)<0$, and tends to infinity as $x\to\infty$. Let $\rho_r> r$ denote the unique real root of $f_r(x)=0$ for $x\geq 1$.

\begin{lemma}\label{lem:critical-gap-general}
For fixed $r\ge2$, $n_r(s)=\rho_r s+o(s)$.  Moreover, for all sufficiently large $s$,
$n_r(s)-n_r(s-1)\ge2$.
\end{lemma}

\begin{proof}
Let 
\[
 \Delta_s(n)\coloneqq\binom nr-\binom{n-s}{r}-\binom{rs+r-1}{r}.
\]
Recall that  $n_r(s)$ is the first admissible integer $n\ge r(s+1)$ for which $\Delta_s(n)\ge0$. Since $\Delta_s(n+1)-\Delta_s(n)=\binom n{r-1}-\binom{n-s}{r-1}>0$,  we have  $\Delta_s(n)$ is
strictly increasing in $n$. Set $n=xs$. Then 
\[
 \Delta_s(n)=\frac{s^r}{r!}\left(x^r-(x-1)^r-r^r\right)+O(s^{r-1}).
\]
By the definition of $\rho_r$, we have $n_r(s)=\rho_r s+o(s)$. 
In the following, we prove that $n_r(s)-n_r(s-1)\geq 2$ for sufficiently large $s$. 
It suffices to show that $\Delta_{s-1}(n_r(s)-2)\geq 0$ for large $s$.
Let $N=n_r(s)$. Using the fact that $\Delta_s(N)\geq 0$, we prove that $\Delta_{s-1}(N-2)-\Delta_s(N)\geq 0$. It follows from $N=\rho_r s+o(s)$ and direct subtraction that 
\begin{align*}
 \Delta_{s-1}(N-2)-\Delta_s(N)
 &=-\binom{N-1}{r-1}-\binom{N-2}{r-1}+\binom{N-s-1}{r-1}
   +\sum_{i=-1}^{r-2}\binom{rs+i}{r-1}  \\
 &=\frac{r^r+(\rho_r-1)^{r-1}-2\rho_r^{r-1}}{(r-1)!}s^{r-1}+o(s^{r-1}).
\end{align*}
Using $r^r=\rho_r^r-(\rho_r-1)^r$, the numerator equals
$(\rho_r-2)\bigl(\rho_r^{r-1}-(\rho_r-1)^{r-1}\bigr)>0$.
Hence, $\Delta_{s-1}(N-2)-\Delta_s(N)> 0$ for large $s$.
\end{proof}

\begin{assumption}\label{ass:link-general}
There is an integer $s_{\rm link}(r)$ such that, for every $t\ge s_{\rm link}(r)$ and every
$m\ge n_r(t)$, every $(r-1)$-uniform hypergraph $J$ on $m$ vertices with $\nu(J)\le t$ satisfies
 \[
 |J|\le \binom m{r-1}-\binom{m-t}{r-1}.
 \]
\end{assumption}
Under Assumption~\ref{ass:link-general}, we can reduce the discussion of Conjecture \ref{Erdos-Matching-Conjecture} for $r$-graphs to a non-uniform hypergraph on   $r^2+r-1$ vertices. 
In the remainder of this section, we will prove how this reduction is achieved, drawing some ideas from \cite{FRR2017}.



For $[n]=\{1,2,\dots,n\}$, given two $r$-sets $A=\{a_1,\ldots,a_r\}$ and
$B=\{b_1,\ldots,b_r\}\in \binom{[n]}r$ with increasingly ordered elements,
write $A\preceq B$ if $a_i\le b_i$ for every $i$, and write $A\prec B$ if
$A\preceq B$ and $A\ne B$.  An $r$-graph
$\mathcal{F}\subseteq\binom{[n]}r$ is \emph{stable} (or \emph{shifted}) if
$A\preceq B$ and $B\in\mathcal F$ imply $A\in\mathcal F$.  The standard
shifting operation produces a stable family
$\operatorname{sh}(\mathcal F)$ with the same size and with
$\nu(\operatorname{sh}(\mathcal F))\le\nu(\mathcal F)$.  Hence it suffices to
prove EMC for stable $r$-graphs.  We
shall also use the following immediate consequence: 
\[
\text{if } \quad \nu(\mathcal{F}) = s \quad \text{and} \quad |\mathcal{F}| = m_r(n, s), \quad \text{then} \quad \nu(sh(\mathcal{F})) = \nu(\mathcal{F}).
\]

\begin{lemma}\label{fact:critical-reduction-r}
Under Assumption~\ref{ass:link-general}, if, for some fixed $s\ge s_{\rm link}(r)$,
\[
m_r(n_r(s)-1,s)=a_r(s)~~\text{and}~~m_r(n_r(s),s)=b_r(n_r(s),s),
\]
then $m_r(n,s)=\max\left\{
a_r(s),
 b_r(n,s)
 \right\}$ for all $n\ge r(s+1)$.
\end{lemma}

\begin{proof}

Using the $r$-graph constructions $\mathcal{A}_r(n,s)$ and $\mathcal{B}_r(n,s)$, we derive the lower bound
\[
m_r(n,s)\ge \max\bigl\{a_r(s),\, b_r(n,s)\bigr\}.
\]
To establish the upper bound for
$r(s+1)\le n\le n_r(s)-1$, add isolated vertices to any admissible
$r$-graph on $n$ vertices until the vertex set has size $n_r(s)-1$.  The
hypothesis at $n_r(s)-1$ then gives
\[
        m_r(n,s)\le m_r(n_r(s)-1,s)=a_r(s).
\]
Hence it remains to show
\begin{equation}\label{EQ:mr_le_br}
m_r(n,s)\le b_r(n,s)
\end{equation}
for all $n\ge n_r(s)$. 
We proceed by induction on $n$. For the base case \(n = n_r(s)\), the desired bound is precisely the statement of the lemma. Now assume that \(n \ge n_r(s) + 1\) and that \eqref{EQ:mr_le_br} holds for every integer \(n'\) satisfying $n_r(s) \le n'< n$.

Let $\mathcal{F}\subseteq\binom{[n]}r$ be an $r$-graph  with $\nu(\mathcal F)=s$. We may assume that $\mathcal F$ is stable. Let  $\mathcal F(\bar n)=\{F\in\F:n\notin\F\}$ and $\mathcal F(n)=\{F\setminus \{n\}:n\in F\in\F\}$. 
Then
\begin{equation*}\label{eq:mr_le_br}
|\mathcal F|
=
|\mathcal F(\bar n)|+|\mathcal F(n)|.
\end{equation*}
Note that $\mathcal F(\bar n)$ is defined on $[n-1]$ and  $\nu(\mathcal F(\bar n))\le s$, so the induction hypothesis gives
\begin{equation}\label{EQ:bound-F-delete-n-induction}
|\mathcal F(\bar n)|\le \max\{a_r(s), b_r(n-1,s)\} =\binom{n-1}{r}-\binom{n-1-s}{r}.
\end{equation}

We claim that $\nu(\mathcal F(n))\le s$.
Suppose, to the contrary, that there are pairwise disjoint edges
\[
e_1,\ldots,e_{s+1}\in \mathcal F(n).
\]
Then $e_i\cup\{n\}\in\mathcal F $ for $i\in [s+1]$.
It follows from $n\ge n_r(s)+1$ and $n_r(s)\ge r(s+1)$ that $n-1\ge r(s+1)$. Consequently, 
\[
\left|[n-1]\setminus\bigcup_{i=1}^{s+1}e_i\right|
\ge
r(s+1)-(r-1)(s+1)
=
s+1.
\]
We can choose distinct vertices $v_1,\ldots,v_{s+1}\in[n-1]\setminus\bigcup_{i=1}^{s+1}e_i$.
As $\mathcal F$ is stable and $v_i<n$,  we have  $e_i\cup\{v_i\}\in\mathcal F$. 
Hence,
\[
\bigl\{e_1 \cup \{v_1\}, \dots, e_{s+1} \cup \{v_{s+1}\}\bigr\}
\]
forms a matching of size \(s+1\) in \(\mathcal F\), contradicting the assumption that \(\nu(\mathcal F) = s\).
Using Assumption~\ref{ass:link-general} gives 
\[
|\mathcal F(n)|\le \binom{n-1}{r-1}-\binom{n-1-s}{r-1}, 
\]
which together with \eqref{EQ:bound-F-delete-n-induction} yields that 
\begin{align*}
|\mathcal F|\le \binom{n-1}{r} -\binom{n-1-s}{r}+ \binom{n-1}{r-1} -\binom{n-1-s}{r-1} =\binom nr-\binom{n-s}{r}= b_r(n,s).
\end{align*}
This completes the proof. 
\end{proof}

We say that a stable $r$-graph $\mathcal{F}\subseteq\binom{[n]}r$  has \emph{property ONE} if $\mathcal{F}$ has a maximum matching avoiding the vertex $1$.
Let
$$m^{ONE}_{r}(n,s)=\max\{|\F|:|V(\F)|=n, \nu(\F)=s,~\text{and}~\F~\text{is stable and has property ONE}\}.$$
An $r$-graph $\F$ with $\nu(\F)=s$ is \emph{maximal} if for every missing $r$-set $Q\notin F$, we have $\nu(\F\cup\{Q\})=s+1$.

\begin{lemma}\label{fact:one-extremizer-maximal}
Let \(\F\subseteq \binom{[n]}{r}\) be stable with \(\nu(\F)=s\),  and have property ONE. If   $|\F|=m^{\rm ONE}_r(n,s)$, then  \(\F\) is maximal. 
\end{lemma}

\begin{proof}
 Suppose, for a contradiction, that there exists an $r$-set \(Q\notin \F\) such that $\nu(\F\cup\{Q\})=s$.
Since shifting does not increase matching number, $\nu\left(\operatorname{sh}(\F\cup\{Q\})\right)\le \nu(\F\cup\{Q\})=s$.
On the other hand, the shifting process can be chosen so that the stable
subfamily $\F$ is retained, and hence
$\nu\left(\operatorname{sh}(\F\cup\{Q\})\right)\ge \nu(\F)=s$.
So $\nu(\operatorname{sh}(\F\cup\{Q\}))=s$.
Moreover, the \(s\)-matching in \(\F\) avoiding vertex \(1\) is still present in \(\operatorname{sh}(\F\cup\{Q\})\). Hence \(\operatorname{sh}(\F\cup\{Q\})\) has property ONE. Therefore
\[
m^{\rm ONE}_r(n,s)\ge |\operatorname{sh}(\F\cup\{Q\})|=|\F|+1,
\]
contradicting the choice of \(\F\). 
\end{proof}

Using a similar  strategy as in the proof of Lemma \ref{fact:critical-reduction-r}, we can obtain the following lemma.

\begin{lemma} \label{fact:ONE-critical-reduction-r}
Under Assumption~\ref{ass:link-general}, if, for some fixed integer \(s\ge s_{\rm link}(r)\), 
\[
m^{\mathrm{ONE}}_r(n_r(s)-1,s)\le a_r(s)
        \quad\text{and}\quad
m^{\mathrm{ONE}}_r(n_r(s),s)\le b_r(n_r(s),s),
\]
then $m^{\mathrm{ONE}}_r(n,s)\le \max\bigl\{a_r(s),\, b_r(n,s)\bigr\}$ for all \(n\ge r(s+1)\).
\end{lemma}

\begin{proof}
 For \(n\le n_r(s)-2\), consider a stable \(r\)-graph
\(\mathcal F\) on \(n\) vertices, with property ONE, with
\(|\mathcal F|=m^{\mathrm{ONE}}_r(n,s)\), and 
\(\nu(\mathcal F)=s\). Adding  isolated vertices
\(n+1,\ldots,n_r(s)-1\) results in an \(r\)-graph
\(\mathcal F'\) which is still stable, has property ONE, and has
\(\nu(\mathcal F')=s\). Thus
\[
m^{\mathrm{ONE}}_r(n,s)=|\mathcal F|
 =|\mathcal F'|
 \le m^{\mathrm{ONE}}_r(n_r(s)-1,s)
 \le a_r(s).
\]
For \(n\ge n_r(s)+1\), we apply induction on \(n\). Assume that $m^{\mathrm{ONE}}_r(n-1,s)\le \max\bigl\{a_r(s), b_r(n-1,s)\}=b_r(n-1,s)$. 
Let \(\F\subseteq \binom{[n]}{r}\) be a stable \(r\)-graph  with property ONE, \(\nu(\mathcal F)=s\), and \(|\mathcal F|=m^{\mathrm{ONE}}_r(n,s)\). Let  $\mathcal F(\bar n)=\{F\in\F:n\notin\F\}$ and $\mathcal F(n)=\{F\setminus \{n\}:n\in F\in\F\}$. 
Then
\begin{equation*}\label{eq:mr_le_br}
|\mathcal F|
=
|\mathcal F(\bar n)|+|\mathcal F(n)|.
\end{equation*}
Using the same argument as  in the proof of Lemma \ref{fact:critical-reduction-r}, we still have $\nu(\mathcal F(n))\le s$ and then $|\mathcal F(n)|\le \binom{n-1}{r-1}-\binom{n-1-s}{r-1}$ under Assumption~\ref{ass:link-general}.  Moreover, start
with an $s$-matching in $\F$ avoiding vertex $1$.  If it uses $n$, then
$n-rs\ge r\ge2$, so there is an unused vertex other than $1$; replacing
$n$ by such a smaller vertex preserves the corresponding edge by stability.
Thus $\F(\bar n)$ has property $\ONE$.  By induction,
$|\F(\bar n)|\le m_r^{\ONE}(n-1,s)\le b_r(n-1,s)$.  Adding the two bounds
completes the proof.
\end{proof}



To prove the next result, we need the following classical theorem due to  Frankl \cite{Frankl2013}. 

\begin{theorem}(Frankl \cite{Frankl2013})\label{input:frankl-large-n}
If $\F$ is an $r$-graph on $n$ vertices with $\nu(\F)\le s$ and $n\ge (2r-1)s+r$, then $|\F|\le \binom nr-\binom{n-s}{r}$.
\end{theorem}

We now transfer the property-$\ONE$ estimate to arbitrary families.  For the
quantitative statement, fix an integer $s_\Delta(r)$ and a real number
$\eta_r$ such that, for every $s\ge s_\Delta(r)$,
\begin{equation}\label{eq:transition-quantitative-assumptions}
        n_r(s)-n_r(s-1)\ge2
        \qquad\text{and}\qquad
        n_r(s)-1\ge \rho_rs+\eta_r.
\end{equation}
Lemma~\ref{lem:critical-gap-general} guarantees the first condition for some
$s_\Delta(r)$.  In the case $r=4$, both conditions are verified
explicitly in Lemma~\ref{lem:r4-explicit-threshold} and in the proof of
Theorem~\ref{thm:main}.

\begin{lemma} \label{lem:one-to-all-r}
Under Assumption~\ref{ass:link-general}, suppose that, for some
$s_0\ge s_\Delta(r)$,
\[
m^{\ONE}_{r}(n,s)\le \max\bigl\{a_r(s),\, b_r(n,s)\bigr\}
\quad\text{for all }s\ge s_0\text{ and }n\ge r(s+1).
\]
Then
\[
m_r(n,s)\le \max\bigl\{a_r(s),\, b_r(n,s)\bigr\}
\quad\text{for all }n\ge r(s+1)
\]
and all
\[
s\ge \max\{s_{\rm link}(r), s_0,c_0s_0+c_1\}.
\]
Here
\[
        c_0=\frac{2r-2}{\rho_r-1},
        \qquad
        c_1=\frac{r-1-\eta_r}{\rho_r-1},
\]
and $s_{\rm link}(r)$ is the constant from
Assumption~\ref{ass:link-general}.
\end{lemma}

\begin{proof}
Suppose, to the contrary, that for some $s\ge \max\{s_{\rm link}(r), s_0,c_0s_0+c_1\}$ and some $n\ge r(s+1)$ we have
$m_r(n,s)>\max\bigl\{a_r(s),\, b_r(n,s)\bigr\}$.
By Lemma~\ref{fact:critical-reduction-r}, there exist $n\in\{n_r(s)-1,n_r(s)\}$
and an $r$-graph \(\F\subseteq \binom{[n]}{r}\) such that $\nu(\mathcal F)=s$ and $|\mathcal F|=m_r(n,s)>\max\bigl\{a_r(s),\, b_r(n,s)\bigr\}$. Let $\mathcal F'=\operatorname{sh}(\mathcal F)$. Then $\mathcal F'$ cannot have property $\ONE$, since otherwise, it follows from $s\ge s_0$ and the the assumption of Lemma~\ref{lem:one-to-all-r} that 
\[
|\mathcal F|=|\mathcal F'|
\le
m_r^{\ONE}(n,s)
\le
\max\bigl\{a_r(s),\, b_r(n,s)\bigr\},
\]
a contradiction. 

Set $\mathcal F'_0 = \mathcal F'$. For $q \ge 1$, let $\mathcal F'_q$ be the induced subhypergraph of $\mathcal F'$ obtained by deleting vertices $1,2,\ldots,q$. Clearly, $\nu(\mathcal F'_q) \le s$ for $q \le s$. It follows from the stability of $\mathcal F'_q$ and the fact that $\mathcal F'$ does not have property $\mathrm{ONE}$ that $\nu(\mathcal F'_q) \ge s - q$ and $\nu(\mathcal F'_1) = s - 1$. Note that $\nu(\mathcal F'_s) \ge 1$; otherwise, we  have $\mathcal F'_s = \emptyset$, which  implies $|\mathcal F| = |\mathcal F'| \le b_r(n,s)$, a contradiction. Therefore, by the pigeonhole principle, there exists some $1 \le q \le s - 1$ such that $\nu(\mathcal F'_q) = \nu(\mathcal F'_{q+1})$, meaning that $\mathcal F'_q$ has property $\mathrm{ONE}$. Define
\[
q_0 = \min\{q \ge 1 : \mathcal F'_q \text{ has property } \mathrm{ONE}\}.
\]
Then $\mathcal F'_{q_0}$ has property $\mathrm{ONE}$, and by the minimality of $q_0$, we have $\nu(\mathcal F'_{q_0}) = s - q_0$.

We next bound $q_0$. 
If $n-q_0\ge (2r-1) (s-q_0)+r$, then by Lemma~\ref{input:frankl-large-n}, we have 
$|\mathcal F'_{q_0}| \le \binom{n-q_0}{r}-\binom{n-s}{r}$.
Note that the number of edges of $\mathcal F'$ meeting $\{1,\ldots,q_0\}$ is at most $\binom nr-\binom{n-q_0}{r}$. 
Therefore
\[
\begin{aligned}
|\mathcal F'|
&\le
\binom nr-\binom{n-q_0}{r}
+
|\mathcal F'_{q_0}|       \\
&\le
\binom nr-\binom{n-q_0}{r}
+
\binom{n-q_0}{r}
-
\binom{n-s}{r}             \\
&=
\binom nr-\binom{n-s}{r}
=
b_r(n,s)
\le
\max\bigl\{a_r(s),\, b_r(n,s)\bigr\},
\end{aligned}
\]
a contradiction. Hence, we have  
\begin{equation}\label{EQ:bound-n-q0-delete-n-induction}
n-q_0\le (2r-1)(s-q_0)+r-1.
\end{equation}
Recall from~\eqref{eq:transition-quantitative-assumptions} that
$n\ge n_r(s)-1\ge \rho_rs+\eta_r$.  Together with
\eqref{EQ:bound-n-q0-delete-n-induction}, this gives
\begin{equation}\label{eq:q0-upper-bound}
q_0\le \frac{(2r-1-\rho_r)s+r-1-\eta_r}{(2r-2)}. 
\end{equation}

We then bound $|\mathcal F'_{q_0}|$ using the  assumption of the lemma. It follows from $n\ge r(s+1)$ that 
\[
n-q_0
\ge
r(s+1)-q_0
=
r(s-q_0+1)+(r-1)q_0
\ge
r(s-q_0+1).
\]

Furthermore, by  \eqref{eq:q0-upper-bound}, we also have 
\begin{align*}
s-q_0 \ge s- \frac{(2r-1-\rho_r)s+r-1-\eta_r}{2r-2} =\frac{\rho_r-1}{2r-2}s -\frac{r-1-\eta_r}{2r-2}.
\end{align*}
By the choice of $c_0$ and $c_1$, $s\ge c_0s_0+c_1$ implies
\[
s-q_0\ge s_0
\]
Since $\mathcal F'_{q_0}$ has property $\ONE$, the assumption of the lemma
gives
\[
|\mathcal F'_{q_0}|
\le
m_r^{\ONE}(n-q_0,s-q_0)
\le
\max\bigl\{a_r(s-q_0), b_r(n-q_0,s-q_0)\bigr\}.
\]

Since \(s-q_0\ge s_0\ge s_\Delta(r)\), \eqref{eq:transition-quantitative-assumptions} gives $n_r(s)-n_r(s-q_0)=\sum_{t=s-q_0+1}^{s}\bigl(n_r(t)-n_r(t-1)\bigr)\ge 2q_0$, which implies that
\begin{align*}
n-q_0\ge n_r(s)-1-q_0
\ge n_r(s-q_0)+2q_0-1-q_0
\ge n_r(s-q_0).
\end{align*}
Hence
\begin{align}\label{EQ:bound- F'_q_0-basis}
|\mathcal F'_{q_0}|
\le
b_r(n-q_0,s-q_0)
\le
\binom{n-q_0}{r}-\binom{n-s}{r}.
\end{align}

To complete the proof, we establish the following bound:
\begin{align}\label{EQ:bound-F'-q-induction-main}
|\mathcal F'_q|
\le
\binom{n-q}{r}-\binom{n-s}{r},
\end{align}
for $0 \le q \le q_0$, by backward induction on $q$. The base case follows from \eqref{EQ:bound- F'_q_0-basis}. Assume now that for some $1 \le q \le q_0$, the inequality \eqref{EQ:bound-F'-q-induction-main} holds. Observe that every edge of $\mathcal F'_{q-1}$ that is not an edge of $\mathcal F'_q$ must contain the vertex $q$. Hence the number of such edges is at most $\binom{n-q}{r-1}$. Consequently,
\begin{align*}
|\mathcal F'_{q-1}|
&\le \binom{n-q}{r-1} + |\mathcal F'_q| \\
&\le \binom{n-q}{r-1} + \binom{n-q}{r} - \binom{n-s}{r} \\
&= \binom{n-q+1}{r} - \binom{n-s}{r}.
\end{align*}
This completes the induction step. Taking $q = 0$, we obtain
\[
|\mathcal F'|
=
|\mathcal F'_0|
\le
\binom{n}{r} - \binom{n-s}{r}
=
b_r(n,s)
\le
\max\bigl\{a_r(s),\, b_r(n,s)\bigr\}.
\]
a contradiction. This completes  the proof.
\end{proof}

Combining Lemmas~\ref{fact:one-extremizer-maximal}, \ref{fact:ONE-critical-reduction-r} and \ref{lem:one-to-all-r}, proving EMC for large $s$, it is enough to prove the following statement: There is a constant $s_0$ such that every stable maximal \(r\)-graph \(\F\subseteq\binom{[n]}r\), with
\begin{align}\label{condiction-of-F}
n\in\{n_r(s)-1,n_r(s)\},\qquad \nu(\F)=s\ge s_0, \qquad \text{   having property ONE},
\end{align}
satisfies $|\F|\le \max\bigl\{a_r(s),\, b_r(n,s)\bigr\}$.

\subsection{The trace and the local board}

By the argument in Subsection \ref{Critical-values}, let $\F\subseteq \binom{[n]}{r}$ be a stable maximal family with matching number $\nu(\F)=s$ and property ONE throughout this subsection.

Put 
\[
N\coloneqq rs+r-1, \quad K\coloneqq [N], \quad  q\coloneqq n-N\ge 1. 
\]
Define the \emph{trace family}  
$$\F_0\coloneqq \{F\cap K: F\in \F\}.$$
Every member in $\F_0$ is called a \emph{trace}. 
We claim that  
\begin{align}\label{EQ:size-trace-F0-least-2}
|F\cap K|\ge 2
\end{align}  
for every \(F\in \mathcal F\).
If \(F\cap K=\emptyset \), then \(F\) is disjoint from the \(s\)-matching \(\mathcal{M}\subseteq K\), giving an \((s+1)\)-matching in \(\mathcal F\), a contradiction.
If \(F\cap K=\{x\}\), by the stability of \(\mathcal F\),  \(F'=(F\setminus\{x\})\cup\{1\}\in \mathcal F\). 
Moreover, property $\ONE$ and stability give  an \(s\)-matching \(\mathcal M\) avoiding vertex \(1\)  with \(V(\mathcal M)\subseteq K\setminus\{1\}\). 
Consequently, \(\{F'\}\cup \mathcal M\) forms an \((s+1)\)-matching in \(\mathcal F\),  a contradiction.

\begin{lemma}\label{lem:trace-matching-general}
$\F_0$ is stable and  $\nu(\F_0)=s$.
\end{lemma}
\begin{proof}
The stability of $\F$ gives the stability of $\F_0$  and $\nu(\F_0)\ge s$. Suppose that \(\mathcal F_0\) contains \(s+1\) pairwise disjoint traces.
Among all choices of pairwise disjoint traces
\(H_1,\ldots,H_{s+1}\in\mathcal F_0\) and edges \(F_1,\ldots,F_{s+1}\in\mathcal F\) satisfying \(F_i\cap K=H_i\) for every \(i\), choose one for which
\[ 
|F_1\cup\cdots\cup F_{s+1}|
\]
is as large as possible.

If the $F_i$ are disjoint, they form an $(s+1)$-matching.
Otherwise two of them meet outside $K$, say in a vertex $u$.
Indeed, if all vertices of $K$ were used, then the $s+1$ edges would have only $r(s+1)-(rs+r-1)=1$ outside position in total, so no outside intersection could occur. Hence, there is a vertex $v\in K\setminus\bigcup_i F_i$. 
Replace $u$ in one edge containing it by $v$.  Stability keeps the new edge in $\F$, the corresponding trace is still disjoint from the other selected traces, and the union of the chosen edges is larger, contradicting the maximality of the choice.  
This contradiction proves $\nu(\F_0)\le s$.
\end{proof}

\begin{lemma}\label{fact:closure-general}
Let $H\in\F_0$ and $|H|<r$.  Then every $r$-set $Q$ with $H\subseteq Q\subseteq[n]$ belongs to $\F$. 
\end{lemma}

\begin{proof}
Suppose that some $r$-set $Q\supset H$ is missing.  By maximality, $\F\cup\{Q\}$ contains an
$(s+1)$-matching using $Q$; hence $\F$ contains $s$ disjoint edges avoiding $Q$.
By $H\in\F_0$, there is an edge $H\cup X_0\in \F$ with $X_0\subseteq[n]\setminus K$.
Let $S=(K\setminus H)\cup X_0$.  After deleting the vertices of those $s$ edges,
at least
\[
        |S|-rs=2r-1-2|H|\ge r-|H|
\]
vertices of $S$ remain. Choose $Y$ of size $r-|H|$ among the smallest such
vertices. Then $H\cup Y\preceq H\cup X_0$, so stability gives
$H\cup Y\in \F$.  This edge is disjoint from the $s$ edges above, contrary
to $\nu(\F)=s$.
\end{proof}



Since \(\mathcal F\) is stable and has property $\ONE$, we can choose an
$s$-matching $\M=\{E_1,\ldots,E_s\}$ avoiding vertex \(1\) such that
\(V(\mathcal M)\subseteq K\setminus\{1\}\).  
Among all such matchings contained in \(K\setminus\{1\}\), choose one for which
$$ R_{\mathcal F}:=K\setminus V(\mathcal{M}) $$
is lexicographically minimal. Then $|R_{\F}|=r-1$.

For each trace \(H\in\F_0\), let \(z(H)=| \{i\in[s]:H\cap E_i\neq\emptyset\}\) be the \emph{spread}  of $H$. The choice of $\M$ and Lemma~\ref{lem:trace-matching-general} implies 
\begin{align}\label{z(H)}
1\le z(H)\le r
\end{align}
For every \(\tau\in\binom{[s]}r\), define the \emph{local board} to be 
\[
V_\tau\coloneqq R_\F\cup(\cup_{i\in\tau}E_i).
\]
Put 
\[
 \F_0^\tau\coloneqq\{H\in\F_0:H\subseteq V_\tau\}. 
\]

Note that the number of local boards \(V_\tau\) containing \(H\) is $\binom{s-z(H)}{r-z(H)}$. On the other hand, by Lemma~\ref{fact:closure-general}, each  \(H\) contributes $\binom{q}{r-|H|}$ edges to \(\F\). We distribute this contribution equally among all local boards containing \(H\). This leads naturally to the weight
\[
w_{s,r}(H)
\coloneqq
\frac{\binom{q}{r-|H|}}
     {\binom{s-z(H)}{r-z(H)}}.
\]
A straightforward double counting  gives
\begin{equation}\label{eq:edge-denote-weight}
|\F|=\sum_{H\in\F_0}\binom{q}{r-|H|}=\sum_{\tau\in\binom{[s]}r}\sum_{H\in\F_0^\tau}w_{s,r}(H).
\end{equation}

Recall that $\mathcal{A}_r(n,s)$ has property ONE, whereas
$\mathcal{B}_r(n,s)$ does not.  We may assume that
$|\F|\ge a_r(s)=|\mathcal{A}_r(n,s)|$, since otherwise the desired bound is
immediate.  It therefore suffices to prove the stronger estimate
$|\F|\le |\mathcal{A}_r(n,s)|$.

To prove this stronger bound, we can also apply \eqref{eq:edge-denote-weight}  to
$\mathcal{A}_r(n,s)$, using the same $s$-matching $E_1,\ldots,E_s$.  Let
\[
        \mathcal A_0:=\{A\cap K:A\in\mathcal A_r(n,s)\},
        \qquad
        \mathcal A_0^\tau:=\{H\in\mathcal A_0:H\subseteq V_\tau\}.
\]
By the symmetry of $\mathcal{A}_r(n,s)$, the corresponding inner sums are
independent of the choice of $\tau$.  Consequently, it is enough to prove,
for every $\tau\in\binom{[s]}r$, the local inequality
\begin{equation}\label{eq:local-board-general}
 \sum_{H\in \F_0^\tau}w_{s,r}(H) \le \sum_{H\in \mathcal{A}_0^\tau}w_{s,r}(H).   
\end{equation}
Summing these inequalities over all $\tau\in\binom{[s]}r$ and using~\eqref{eq:edge-denote-weight} gives $|\F|\le |\mathcal{A}_r(n,s)|$. Hence the problem reduces naturally to a finite-board comparison.

\begin{theorem}\label{thm:finite-board-criterion}
Suppose that Assumption~\ref{ass:link-general} holds, and let $s_\Delta(r)$ and $\eta_r$
satisfy~\eqref{eq:transition-quantitative-assumptions}.  
Suppose that there is an integer $s_0\ge \max\{s_\Delta(r),s_{\mathrm{link}}(r)\}$ such that, for every $s\ge s_0$, every $n\in\{n_r(s)-1,n_r(s)\}$, and every stable maximal property-$\ONE$ $r$-graph $\F\subseteq\binom{[n]}r$ with $\nu(\F)=s$, the local board
inequality~\eqref{eq:local-board-general} holds for every
$\tau\in\binom{[s]}r$.  Then the EMC holds for all
$n\ge r(s+1)$ and all integers
\[
        s\ge \max\{s_{\rm link}(r),s_0,c_0s_0+c_1\},
\]
where $c_0$ and $c_1$ are as in Lemma~\ref{lem:one-to-all-r}.
\end{theorem}

\begin{proof}
Fix $s\ge s_0$ and $n\in\{n_r(s)-1,n_r(s)\}$.  By
Lemma~\ref{fact:one-extremizer-maximal}, an extremizer for
$m_r^{\ONE}(n,s)$ may be taken to be stable and maximal.  Summing the assumed
local inequalities~\eqref{eq:local-board-general} and using~\eqref{eq:edge-denote-weight} gives
\[
        m_r^{\ONE}(n,s)\le a_r(s).
\]
At $n=n_r(s)-1$ this is the required clique-side bound, while at
$n=n_r(s)$ it is at most $b_r(n_r(s),s)$ by the definition of $n_r(s)$.  Lemma~\ref{fact:ONE-critical-reduction-r} therefore extends the
property-$\ONE$ estimate to every $n\ge r(s+1)$.  The conclusion now follows
from Lemma~\ref{lem:one-to-all-r}, together with the two constructions giving
the matching lower bound.
\end{proof}

At the end of section, we remark that the competing family $\mathcal{B}_r(n,s)$ is not discarded by this
reduction: stable families without property $\ONE$ are handled by the
deletion chain in Lemma~\ref{lem:one-to-all-r}.  Thus the finite-board
analysis below only needs to treat the $\mathcal{A}_r$-type,
property-$\ONE$ branch.

\section{The $19$-board setup  and proof strategy}\label{sec:board-setup}
For the remainder of the paper, we specialize to the case $r=4$. 
Let $\mathcal{F}\subseteq \binom{[n]}{4}$ be a stable maximal family 
with matching number $\nu(\mathcal{F})=s$ and property ONE.
Note that
\[
        a_4(s)\coloneqq\binom{4s+3}4,
        \qquad
        b_4(n,s)\coloneqq\binom n4-\binom{n-s}4,
        \qquad
        n_4(s)\coloneqq\min\{n\ge4(s+1):b_4(n,s)\ge a_4(s)\}.
\]
\begin{theorem}(Frankl, R\H{o}dl and Ruci\'{n}ski~\cite{FRR2017}) \label{lem:FRR2017}
For all $s\ge 33$ and $n\ge 3s+3$, $m_3(n,s)=\max\{a_3(s), b_3(n,s)\}$.
\end{theorem}
Indeed, if $t\ge33$ and $n\ge n_4(t)$, then $n\ge4t+4$, and $b_3(4t+4,t)-a_3(t)
=\frac{t(5t^2+18t+10)}{3}>0$.
Since $b_3(n,t)$ is increasing in $n$, Theorem~\ref{lem:FRR2017} yields $
m_3(m,t)=b_3(m,t)$,
which is precisely Assumption~\ref{ass:link-general} for $r=4$.
For $H\in \F_0$, it follows from \eqref{EQ:size-trace-F0-least-2} that $|H| \ge 2$. We  write $ w(H)=\frac{\binom q{4-|H|}}{\binom{s-z(H)}{4-z(H)}}$ for the weight $w_{s, 4}(H)$ defined in Section~\ref{sec:general-reduction}.

\begin{lemma}\label{lem:r4-explicit-threshold}
For every integer \(s\ge 3481\), we have 
\[
n_4(s)\le 4s+3+\frac{12}{25}(s-3),~n_4(s)-1\ge 4s+4~\text{and}~n_4(s)-n_4(s-1)\ge 2.
\]
Consequently, if $n\in\{n_4(s)-1,n_4(s)\}$, $q:=n-(4s+3)$, then
\[
        \frac{q}{s-3}\le \frac{12}{25}.
\]
In particular, the following weight estimates hold:
\begin{align*}
        |H|=3,\ z(H)=3
        &\implies
        w(H)\le \frac{12}{25},\\
        |H|=2,\ z(H)=2
        &\implies
        w(H)\le \frac{144}{625},\\
        |H|=3,\ z(H)=2
        &\implies
        (s-3)w(H)\le \frac{24}{25},\\
        |H|=2,\ z(H)=1
        &\implies
        (s-3)w(H)\le \frac{432}{625},\\
        |H|=3,\ z(H)=1
        &\implies
        \binom{s-2}{2}w(H)\le \frac{36}{25}.
\end{align*}
\end{lemma}

\begin{proof}
For real \(x\), define $\Delta_s(x):= \binom{x}{4}-\binom{x-s}{4}-\binom{4s+3}{4}$.
Set
\[
        m(s):=
        \left\lfloor \frac{112s+39}{25}\right\rfloor=\frac{112s+39-\varepsilon_s}{25},
        \qquad
        0\le \varepsilon_s\le 24.
\]
We first prove that $\Delta_s(m(s))\ge 0$, which implies 
\begin{equation*}
        n_4(s)\le m(s)
        \le
        \frac{112s+39}{25}
        =
        4s+3+\frac{12}{25}(s-3).
\end{equation*}
Observe that \[
\Delta_s'(x)= \frac{s}{12} \left( 6x^2-6sx+2s^2-18x+9s+11\right)
\] 
is increasing and greater than $0$ for \(x\ge 4s+4\). So, \(\Delta_s(x)\) is increasing.

If \(\varepsilon_s\le 23\), then $m(s)\ge \frac{112s+16}{25}>4s+4$.
Hence
\begin{align*}
     \Delta_s(m(s))\ge \Delta_s\left(\frac{112s+16}{25}\right)=\frac{s}{375000}P_1(s),
\end{align*}
where $P_1(s) =2487s^3-8567702s^2-2509011s-347566$.
It follows from $P_1'(s)>0$  for $s\ge 3481$ and $P_1(3481)>0$ 
that $\Delta_s(m(s))\ge 0$.

Otherwise, we have  \(\varepsilon_s=24\), and then 
\[
        112s+39\equiv 24 \pmod {25}.
\]
This is equivalent to $ s\equiv 5\pmod {25}$.
Thus, under \(s\ge 3481\), one has \(s\ge 3505\). 
\[
        \Delta_s(m(s))
        =
        \Delta_s\left(\frac{112s+15}{25}\right)
        =
        \frac{s}{125000}P_2(s),
\]
where $P_2(s) =829s^3-2895710s^2-818825s-116750$.
It follows from $ P_2'(s)>0$ for $s\ge 3505$ and $P_2(3505)>0$ that  $\Delta_s(m(s))\ge 0$.

Next, by 
$$a_4(s)-b_4(4s+4,s)=\binom{4s+3}{4}-\left(\binom{4s+4}{4}-\binom{3s+4}{4}\right)=\frac{s(s-1)(81s^2+95s+26)}{24}>0.$$
Hence 
\[
n_4(s)-1\ge 4s+4.
\]

Now we prove $n_4(s) - n_4(s-1) \ge 2$. 
Let $N:=n_4(s)$. Then $\Delta_s(N)\ge 0$. As in the proof of Lemma  \ref{lem:critical-gap-general}, it 
suffices to show that $\Delta_{s-1}(N-2)- \Delta_s(N)> 0$.
A direct subtraction gives
\[
\begin{aligned}
        6\bigl(\Delta_{s-1}(N-2)-\Delta_s(N)\bigr)
        =&-N^3-3N^2s+9N^2+3Ns^2+12Ns      \\
        &-26N+255s^3-102s^2+45s+18.
\end{aligned}
\]
Denote the right-hand side by \(F_s(N)\). For fixed \(s\),
\begin{align*}
        \frac{\partial F_s}{\partial N}
        =-3N^2-6sN+18N+3s^2+12s-26<0~~\text{for}~~N\ge 4s+4,
\end{align*}

Therefore \(F_s(N)\) is decreasing for \(N\ge 4s+4\).
Recall that $N=n_4(s)\le \frac{112s+39}{25}$.
We have 
\[
        F_s(N)
        \ge
        F_s\left(\frac{112s+39}{25}\right) =
        \frac{
        1848647s^3+18927s^2+516094s-69594
        }{15625}>0, 
\]
and then $\Delta_{s-1}(N-2)-\Delta_s(N)>0$.

Now let $n\in\{n_4(s)-1,n_4(s)\}$, $q:=n-(4s+3)$. Then $q\le  n_4(s)-4s-3\le\frac{12}{25}(s-3)$,
and therefore
\[
        \frac{q}{s-3}\le \frac{12}{25}.
\]
The weight estimates for $H$ follow directly from the following obvious inequalities.
\begin{align*}
        \frac{q}{s-3}&\le \frac{12}{25};\\
        \frac{\binom q2}{\binom{s-2}{2}}&=\frac{q(q-1)}{(s-2)(s-3)}\le \left(\frac{12}{25}\right)^2;\\
        (s-3)\frac{q}{\binom{s-2}{2}}&=\frac{2q}{s-2}\le 2\cdot \frac{12}{25};\\
        (s-3)\frac{\binom q2}{\binom{s-1}{3}}&=\frac{3q(q-1)}{(s-1)(s-2)}\le 3\left(\frac{12}{25}\right)^2;\\
        \binom{s-2}{2}\frac{q}{\binom{s-1}{3}} &=\frac{3q}{s-1}\le 3\cdot\frac{12}{25}.
\end{align*}
This completes the proof. 
\end{proof}

We retain all the notation from the preceding discussion of $r$-graphs, specialized throughout to the case $r=4$.
Write $E_i=\{e_{i,1}<e_{i,2}<e_{i,3}<e_{i,4}\}$ and $R_{\F}= K\setminus\bigcup_{i=1}^sE_i=\{1,\delta_1,\delta_2\}$ with $\delta_1<\delta_2$.
For $i\in[s]$ and $j\in[4]$, we call $R_{\mathcal F}\cup\{e_{i,j}\}$ the \emph{\(e_{i, j}\)-seed}.
The next lemma shows that the $R_{\mathcal F}\cup\{e_{i,u}\}$ are always present for $u\in\{1,2\}$; we refer to them as the seeds.

The following statement requires $q\ge1$.  This is exactly the situation in the critical applications: if $s\ge3481$ and $n\in\{n_4(s)-1,n_4(s)\}$, then Lemma~\ref{lem:r4-explicit-threshold} gives $n\ge4(s+1)$ and hence $q=n-(4s+3)\ge1$.

\begin{lemma}\label{lem:seeds-r4}
Assume $q\ge1$.  For every $i$, $R_{\F}\cup\{e_{i,1}\}$ and $R_{\F}\cup\{e_{i,2}\}$ belong to $\F_0$.
\end{lemma}
\begin{proof}
If $e_{i,2}>\delta_2$, then the increasing list of $R_{\F}\cup\{e_{i,2}\}$ is coordinatewise no larger than that of $E_i$,
because
\[
        1< e_{i,1},\qquad \delta_1<e_{i,2},\qquad \delta_2<e_{i,2}<e_{i,3},\qquad e_{i,2}<e_{i,4}.
\]
Thus stability gives $R_{\F}\cup\{e_{i,2}\}\in \F_0$.

Assume $e_{i,2}<\delta_2$.  We first show that $\{1,\delta_1,e_{i,2}\}\in \F_0$.  
Suppose not.  Since $q \ge 1$, choose an outside vertex $u\notin K$.  
Then $\{1,\delta_1,e_{i,2},u\}$ is not an edge.  
By maximality, $\F\cup\{1,\delta_1,e_{i,2},u\}$ contains an $(s+1)$-matching using $\{1,\delta_1,e_{i,2},u\}$, so $\F$ contains an $s$-matching disjoint from $\{1,\delta_1,e_{i,2}\}$. List the \(4s\) vertices used by this matching increasingly as $u_1<\cdots<u_{4s}$.

We now push this matching into \(K\setminus\{1,\delta_1,e_{i,2}\}\). 
List the vertices of \(K\setminus\{1,\delta_1,e_{i,2}\}\) increasingly as $v_1<\cdots<v_{4s}$.
Since \(K\setminus\{1,\delta_1,e_{i,2}\}\) consists of the \(4s\) smallest vertices of
\([n]\setminus\{1,\delta_1,e_{i,2}\}\), we have $v_j\le u_j$ for every $j$.
Replace each used vertex \(u_j\) by \(v_j\). For each edge of the matching, the resulting \(4\)-set is
coordinatewise no larger than the original edge; hence it is still an edge of \(\F\) by stability. The new \(s\)-matching uses exactly the vertices of $K\setminus\{1,\delta_1,e_{i,2}\}$.
Thus it is an admissible property-ONE \(s\)-matching contained in \(K\), whose remaining vertices are
$\{1,\delta_1,e_{i,2}\}$.
This contradicts the choice of \(E_1,\ldots,E_s\), for which $\{\delta_1,\delta_2\}$ was chosen lexicographically minimal.
Hence $\{1,\delta_1,e_{i,2}\}\in \F_0 $.
By Lemma~\ref{fact:closure-general} every \(4\)-set containing this trace is an edge of \(\F\). In particular,
\[
R_{\F}\cup\{e_{i,2}\}
=
\{1,\delta_1,\delta_2,e_{i,2}\}
\in \F.
\]
Since this set is contained in \(K\), we have $R_{\F}\cup\{e_{i,2}\}\in \F_0 $.
Finally, $R_{\F}\cup\{e_{i,1}\}\preceq R_{\F}\cup\{e_{i,2}\}$. Stability gives $R_{\F}\cup\{e_{i,1}\}\in \F$.
Again this \(4\)-set is contained in \(K\), so $R_{\F}\cup\{e_{i,1}\}\in \F_0 $.
This completes the proof of the lemma. 
\end{proof}

\begin{lemma}\label{lem:local-matching-r4}
For $\tau\in\binom{[s]}4$, the  $19$-board $V_{\tau}$ does not contain five pairwise disjoint traces from $\F_0$.
\end{lemma}
\begin{proof}
Otherwise, five disjoint traces in $V_\tau$, together with the $s-4$ matching edges outside the four selected edges, would give $s+1$ pairwise disjoint members of the mixed trace family $\F_0$, contradicting
Lemma~\ref{lem:trace-matching-general}.
\end{proof}

\paragraph{Proof strategy.} To prove Theorem \ref{thm:main},  it suffices to establish the local comparison~\eqref{eq:local-board-general} on each \(19\)-board at the two critical values of $n$ by Theorem~\ref{thm:finite-board-criterion}. 
For the family \(\mathcal A_4(n,s)\), every trace on \(K\)
has size four. A general family \(\F\), however, may also have traces
of size two or three. By Lemma~\ref{fact:closure-general}, each such
smaller trace has many completions outside \(K\), and hence contributes
additional weight. These smaller traces cannot occur freely: if too
many \(4\)-subsets of the same board were also present, they would
produce five pairwise disjoint traces, contradicting
Lemma~\ref{lem:local-matching-r4}. Thus the presence of smaller traces
forces some \(4\)-subsets of the board to be absent. Our goal is to
prove the compensation inequality
\begin{equation}\label{EQ:main-inequality-4-uniform}
  \sum_{\substack{H\subseteq V_\tau, \ H\in \F_0,\ \\ |H|=2,3}} w(H)
 \le  \sum_{\substack{H_Q\subseteq V_\tau,\ H_Q\notin \F_0\\ |H_Q|=4}} w(H_Q).   
\end{equation}
We organize the proof according to the number of selected matching edges met
by a missing trace of size 4. The part meeting the largest number of matching edges is treated
on the full \(19\)-board (Layer $1$). The remaining parts are handled on the
subboards formed by \(R_{\F}\) together with three or two selected
matching edges, which have \(15\) (Layer $2$) and \(11\) vertices (Layer $3$), respectively. 
The next three sections establish these local estimates. Their sum gives the displayed compensation inequality, and the weighted double-counting identity then yields
\[
|\F|\le |\mathcal A_4(n,s)|.
\]

\section{Layer $1$: the $19$-board inequality for wide triples and wide pairs}\label{sec:layer1}
Fix $\tau$ with $|\tau|=4$.  Without loss of generality, we may assume that $\tau=\{1,2,3,4\}$. Then  $V_\tau=R_{\F}\cup E_1\cup E_2\cup E_3\cup E_4$.
We refer to \(E_k\) as the \(k\)-th selected column, regard \(e_{k,j}\) as the vertex at level \(j\) of \(E_k\), and call \(k\) the \emph{support} of \(E_k\).
For a trace $H\in \F_0$, the \emph{support} of $H$ is 
\[
 \mathrm{supp}(H)\coloneqq
\{k\in[4]:H\cap {E_{k}}\neq\emptyset\}, 
\]
We call $H$ is \emph{quad} if $|H|=4$, \emph{triple} if $|H|=3$, and  \emph{pair} if $|H|=2$. 
A  trace $H\in \F_0$ is called \emph{wide} if it contains at most one vertex from each selected column and contains no vertex of \(R_{\F}\). 
For a subset $H$, if $H\in \F_0$, we call it \emph{present}; otherwise, we call it \emph{missing}. Let $Q_4\subseteq[4]^4$ be the set of missing wide quads of $\F_0$.  
Stability says that the present wide quads form a down-set: replacing a selected-column level by a smaller level in the same column gives, after sorting, a $4$-set coordinatewise no larger than the original one.  Hence $Q_4$ is an up-set.  

Let $J=\{j_1<j_2\}\in \binom{[4]}{2}$ be a pair support, and a corresponding pair level \(\mathbf{u}=(u_1,u_2)\in [4]^2\). Set 
\[
H_J(\mathbf{u}):=\{e_{j_1,u_1}, e_{j_2,u_2}\}, \qquad 
P_J:=\{\mathbf{u}\in [4]^2: H_J(\mathbf{u})\in \F_0\}.
\]
Let $P_2\coloneqq\{H_J(\mathbf{u}): \mathbf{u}\in P_J, J\in \binom{[4]}{2}\}$,
The corresponding definitions for triples are analogous. In particular, for a triple support $I\in\binom{[4]}3$, let 
\[
       T_I:=\{\mathbf{u}'\in [4]^3: H_I(\mathbf{u}')\in \F_0\}, \qquad  
       T_3\coloneqq \left\{    H_I(\mathbf{u}'): \mathbf{u}'\in T_I, I\in \binom{[4]}{3}  \right\}
\]
Observe that by Lemma \ref{lem:r4-explicit-threshold}, we have 
\[
w(H)\le \frac{12}{25} \text{ if } H\in T_3; \quad \text{ and } w(H)\le \frac{144}{625} \text{ if } H\in P_2. 
\]
The following is the main result of this section. 
\begin{theorem} \label{THM:Lay1-main-result}
On every $19$-board, 
 \begin{equation}\label{eq:layer1-target}
        \frac{12}{25}|T_3|+\frac{144}{625}|P_2|\le |Q_4|.
\end{equation}   
\end{theorem}

\subsection{Bounds for pairs and triples}
In this subsection, we prove $|P_2|\le 24$ and $|T_3|\le 128$.

\begin{lemma}\label{lem:pair-ferrers}
For every pair support \(J\), the set of present wide pair levels $P_J$ is a down-set in \([4]^2\) of size at most \(4\). Consequently
\[
        |P_2|\le 24.
\]
\end{lemma}

\begin{proof}
Fix a pair support \(J=\{j_1,j_2\}\). 
First we prove that \(P_J\) is a down-set. Suppose \(\mathbf{x}=(x_1,x_2)\in P_J\) and \(\mathbf{y}=(y_1,y_2)\) satisfies \(1\le y_1\le x_1\) and \(1\le y_2\le x_2\). 
Since \(\mathbf{x}\in P_J\), there is an edge \(F\in \F\) whose exact trace on \(K\) is \(H_J(\mathbf{x})\).
Thus $F=H_J(\mathbf{x})\cup X$ for some \(X\subseteq [n]\setminus K\). Put
\[
        F':=H_J(\mathbf{y})\cup X.
\]
Since \(\F\) is stable, \(F'\in \F\). Moreover \(F'\cap K=H_J(\mathbf{y})\), and hence \(H_J(\mathbf{y})\in \F_0\). Therefore \(\mathbf{y}\in P_J\). Thus,  \(P_J\) is  a down-set.
We next record the local obstruction that will be used below. 
\begin{claim}\label{pair-existence}
    There do not exist two levels $\mathbf{x}=(x_1,x_2)$, $\mathbf{y}=(y_1,y_2)$ in \(P_J\) such that $x_1\neq y_1$, $x_2\neq y_2$ and $\{x_r, y_r\}\ne \{1, 2\}$ for some \(r\in\{1,2\}\).
\end{claim}
\begin{proof}
    If such \(\mathbf{x}, \mathbf{y}\) existed, then there would be a level $z\in \{1,2\}\setminus \{x_r, y_r\}$.
By Lemma~\ref{lem:seeds-r4}, the seed $R_{\F}\cup \{e_{j,z}\}$ belongs to \(\F_0\) for $j\in J$. Since \(z\notin \{x_r, y_r\}\), this seed is disjoint from both \(H_J(\mathbf{x})\) and \(H_J(\mathbf{y})\). 
Also \(H_J(\mathbf{x})\) and \(H_J(\mathbf{y})\) are disjoint from each other, because their two coordinates are different.
Hence, the two columns outside $J$, together with $R_{\F}\cup \{e_{j,z}\}$, $H_J(\mathbf{x})$, and $H_J(\mathbf{y})$, form five pairwise disjoint members of $\F_0$ within the same $19$-board.
This contradicts Lemma~\ref{lem:local-matching-r4}. 
\end{proof}
It remains to prove  \(|P_J|\le 4\). Since \(P_J\) is a   down-set in \([4]^2\), it is represented by a height vector $4\ge h_1\ge h_2\ge h_3\ge h_4\ge 0$ such that
\[
        P_J=\{(x,y):1\le x\le 4,\ 1\le y\le h_x\}.
\]
Thus $|P_J|=h_1+h_2+h_3+h_4$. We distinguish three cases.

First suppose \(h_1\ge 3\). Then \((1,3)\in P_J\). If \(h_2\ge 1\),
then \((2,1)\in P_J\). The two levels \((1,3)\) and \((2,1)\) contradict Claim~\ref{pair-existence}.
By monotonicity of the height vector, also \(h_3=h_4=0\), and consequently
\[
        |P_J|=h_1\le 4.
\]
Next suppose \(h_1=2\). Then \((1,2)\in P_J\). If \(h_3\ge 1\), then \((3,1)\in P_J\) contradict Claim~\ref{pair-existence}.
Hence \(h_3=h_4=0\), and since \(h_2\le h_1=2\), we get
\[
        |P_J|=h_1+h_2\le 2+2=4.
\]
Finally suppose \(h_1\le 1\). Since $h_1\ge h_2\ge h_3\ge h_4$, all four heights are at most \(1\). Hence
\[
        |P_J|=h_1+h_2+h_3+h_4\le 4.
\]
In all cases, \(|P_J|\le 4\). Since the support \(J\) was arbitrary and there are exactly \(\binom{4}{2}=6\) pair supports, we obtain
\begin{align*}
    |P_2|=\sum_J |P_J|\le 6\cdot 4=24.
\end{align*}
This completes the proof of Lemma \ref{lem:pair-ferrers}. 
\end{proof}

A $3$-matching $\mathbf{a}=(a_1, a_2, a_3), \mathbf{b}=(b_1,b_2,b_3), \mathbf{c}=(c_1,c_2,c_3)\in[4]^3$ is called \emph{bad} if, in every coordinate, the three levels $a_j,b_j,c_j$ are pairwise distinct, and in at least one coordinate, the unused level in $[4]\setminus\{a_j,b_j,c_j\}$ is $1$ or $2$.

\begin{lemma}\label{lem:triple-ferrers}
For a  triple support $I\in\binom{[4]}3$, we have $T_I$ is a down-set, does not contain bad \(3\)-matching and $|T_I|\le 32$. Consequently, $|T_3|\le 128$.
\end{lemma}

\begin{proof}

Using the same argument as in the proof of Lemma~\ref{lem:pair-ferrers}, we have $T_I$ is a down-set; we omit the details.

We first show that $T_I$ contains no bad $3$-matching. If three present
wide triple levels formed one, then in some supporting column \(k\) they
would omit level $t\in \{1,2\}$. By Lemma~\ref{lem:seeds-r4}, the corresponding
seed $R_{\F}\cup\{e_{k,t}\}$ is present. This seed, the three triples, and
the fourth full column would be five pairwise disjoint traces of $\F_0$, contradicting
Lemma~\ref{lem:local-matching-r4}.

Next we prove $|T_I|\le 32$. Put
\[
        S:=\{1,2\},
        \qquad
        L:=\{3,4\}.
\]
Suppose that \(T_I\cap L^3\neq\emptyset\). Since  \(T_I\) is a down-set,  $(3,3,3)\in T_I$. We claim that $T_I\subseteq \{1,2,3\}^3$ and so 
\[
|T_I|\le 27<32.
\]
Indeed, if some point of \(T_I\) has a coordinate equal to \(4\), then after
permuting coordinates and lowering the other two coordinates, the down-set
property gives \((1,1,4)\in T_I\). Since \((3,3,3)\in T_I\), it also gives
\((2,2,1)\in T_I\).  But $(1,1,4), (2,2,1), (3,3,3)$ form a bad \(3\)-matching, which is impossible.

In the following, assume that $T_I\cap L^3=\emptyset$. For \(i\in\{1,2,3\}\), let
\[
        T_{L^2}^i:=\{\mathbf{x}\in[4]^3:x_i\in S,\ x_j\in L\text{ for }j\neq i\},
\]
\[
       T_{S^2}^i:=\{\mathbf{x}\in[4]^3:x_i\in L,\ x_j\in S\text{ for }j\neq i\},
\]
Then the sets
\[
        S^3, \quad  T_{L^2}^1,  T_{L^2}^2,  T_{L^2}^3, \quad T_{S^2}^1, T_{S^2}^2, T_{S^2}^3
\]
are pairwise disjoint and contain all possible points of \(T_I\). 
\begin{claim}\label{claim:relationship-related-to-T}
    For each \(i\), $|T_I\cap T_{L^2}^i| \le |T_{S^2}^i\setminus T_I|$.
\end{claim}
\begin{proof}
 By symmetry it suffices to prove this for \(i=3\). Write $T_{L^2}^3=L\times L\times S$ after permuting coordinates if necessary. 
 For notational convenience in the verification below, we prove the equivalent coordinate form $|T_I\cap (L\times L\times S)| \le|(S\times S\times L)\setminus T_I|$.

If \(T_I\cap (L\times L\times S)=\emptyset\), there is nothing to prove. Suppose first that $1\le |T_I\cap (L\times L\times S)|\le 4$.
The down-set property gives $(3,3,1)\in T_I$.
Consider the following four disjoint pairs in \(S\times S\times L\):
\begin{align*}
      \{(1,1,3),(2,2,4)\};~~ \{(1,1,4),(2,2,3)\};~~\{(1,2,3),(2,1,4)\};~~\{(1,2,4),(2,1,3)\}.
\end{align*}
Each of the above pairs forms a bad $3$-matching with \((3,3,1)\). 
Hence every pair cannot be entirely contained in \(T_I\), which implies that $|(S\times S\times L)\setminus T_I|\ge 4\ge |T_I\cap (L\times L\times S)|$.

In the following, suppose that $|T_I\cap (L\times L\times S)|\ge 5$.
Note that \(T_I\cap (L\times L\times S)\) is a down-set inside \(L\times L\times S\). If \(T_I\cap (L\times L\times S)\) contains \((3,4,2)\), then every point \(q\in S\times S\times L\) must be absent from \(T_I\). 
Indeed, if such a \(q\) were present, then the down-set property would give
\[
        (1,1,3)\in T_I,
        \qquad
        (2,3,1)\in T_I,
\]
and the three points  $(3,4,2)$, $(1,1,3)$ and $(2,3,1)$ would form a bad \(3\)-matching. 
Thus \(S\times S\times L\subseteq  (S\times S\times L)\setminus T_I\), and so \(|(S\times S\times L)\setminus T_I|=8\ge |T_I\cap (L\times L\times S)|\). 
The same argument, with the first two coordinates interchanged, applies if \(T_I\cap (L\times L\times S)\) contains \((4,3,2)\). If \(T_I\cap (L\times L\times S)\) contains \((4,4,2)\), then it also contains \((3,4,2)\), so this case is already covered.

We may therefore assume that none of $(3,4,2), (4,3,2), (4,4,2)$ lies in \(T_I\cap (L\times L\times S)\). 
Since \(|T_I\cap (L\times L\times S)|\ge 5\), the only possible point of \(T_I\cap (L\times L\times S)\) in \(L\times L\times\{2\}\) is \((3,3,2)\), and the whole bottom layer \(L\times L\times\{1\}\) must be contained in \(T_I\cap (L\times L\times S)\). 
In particular,
\[
        (3,3,2)\in T_I,
        \qquad
        (3,4,1)\in T_I.
\]
Now let \(q\in S\times S\times L\). 
If \(q\in T_I\), then \((1,1,3)\in T_I\). Also, from \((3,3,2)\in T_I\), we get \((2,3,2)\in T_I\). Hence
\[
        (1,1,3),\quad (2,3,2),\quad  (3,4,1)
\]
form a bad \(3\)-matching. Therefore again $S\times S\times L\subseteq (S\times S\times L)\setminus T_I$, so $|(S\times S\times L)\setminus T_I|=8\ge |T_I\cap (L\times L\times S|$.
This completes the proof. 
\end{proof}
By  \(T_I\cap L^3=\emptyset\) and Claim \ref{claim:relationship-related-to-T},  We conclude that 
\begin{align*}
        |T_I|&=|T_I\cap S^3|+\sum_{i=1}^3 |T_I\cap T_{L^2}^i|+\sum_{i=1}^3 |T_I\cap  T_{S^2}^i| \\
        &\le|T_I\cap S^3| +\sum_{i=1}^3 |T_{S^2}^i\setminus T_I| +\sum_{i=1}^3 |T_I\cap T_{S^2}^i| \\
        &=|T_I\cap S^3|+\sum_{i=1}^3 |T_{S^2}^i|\\
        &\le|S^3|+3\cdot 8=32.
\end{align*}

There are four triple supports in a \(19\)-board. Applying the bound
\(|T_I|\le 32\) to each support $I$ gives
\[
        |T_3|\le 4\cdot 32=128.
\]
\end{proof}

It follows from Lemmas \ref{lem:pair-ferrers} and \ref{lem:triple-ferrers} that 
\begin{equation}\label{eq:T3-P2-M}
      \frac{12}{25}|T_3|+\frac{144}{625}|P_2|
        \le \frac{12}{25}\cdot128+\frac{144}{625}\cdot24
        =66.9696<67.
\end{equation}

Hence \eqref{eq:layer1-target} holds whenever $|Q_4|\ge67$. It remains to consider the case in which \(|Q_4|\le 66\). We divide this case into two branches: the top-star branch and the no-top-star branch.

\subsection{The top-star branch: an exact Farkas certificate}

A \emph{full top-star in direction \(4\)} is $S_4^+=\{(x_1,x_2,x_3,4):x_i\in\{2,3,4\}\}$.

\begin{lemma}\label{lem:topstar}
Assume that the missing wide quads set $Q_4$ contains a full top-star.  
After permuting the four selected columns, assume that $S^+_4\subseteq Q_4$.
If \(|Q_4|\le 66\), then 
 \begin{equation*}
        \frac{12}{25}|T_3|+\frac{144}{625}|P_2|\le |Q_4|.
\end{equation*} 
\end{lemma}

\begin{proof}
Let
\[
        C^{\rm top}:=\{x\in[4]^3:(x,4)\notin Q_4\},~~~~~~ N:=Q_4\cap([4]^3\times\{1,2,3\}).
\]
Write $c:=|C^{\rm top}|$, $\ell:=|N|$. Obviously,  
\begin{equation}\label{EQ:M=64-c+l}
  |Q_4|=64-c+\ell.  
\end{equation}
Moreover, \(S_4^+\subseteq Q_4\) means that $c\le 4^3-3^3=37$.
If, in addition, \(|Q_4|\le66\), then $64-c+\ell=|Q_4|\le66$, so
\[
        \ell\le c+2.
\]
We introduce the labelled variables used in the top-star relaxation.
For \(x\in [4]^3\), let
\begin{align*}
    y_x:=\mathds{1}_{\{x\in C^{top} \}},~ m_{x,t}:=\mathds{1}_{\{(x,t)\in N\}}, 
\end{align*}
that is, \(y_x=1\) means that the wide quad \((x,4)\) is present, and \(m_{x,t}=1\) means that the wide quad \((x,t)\) is missing.
For a triple support \(I=\{i_1<i_2<i_3\}\in\binom{[4]}3\) and a level  $\mathbf{v}=(v_1,v_2,v_3)\in[4]^3$, recall 
\[
        H_I(\mathbf{v}):=\{e_{i_1,v_1},e_{i_2,v_2},e_{i_3,v_3}\},  \quad T_I := \{\mathbf{v}\in[4]^3:H_I(\mathbf{v})\in \F_0\}.
\]
Define 
\[
        t_{I,\mathbf{v}}:=\mathds{1}_{\{\mathbf{v}\in T_I\}}. 
\]

For \(J=\{j_1, j_2\}\in\binom{[4]}2\) and $\mathbf{u}=(u_1,u_2)\in [4]^2$, recall 
$H_J(\mathbf{u})=\{e_{j_1,u_1},e_{j_2,u_2}\}$ and $P_J= \{\mathbf{u}\in[4]^2:H_J(\mathbf{u})\in \F_0\}$. Define 
\[
        p_{J,\mathbf{u}}:=\mathds{1}_{\{\mathbf{u}\in P_J\}}.
\]

Let \(X\) denote the vector consisting of all variables
\[
        y_x,\quad m_{x,t},\quad t_{I,\mathbf{v}},\quad p_{J,\mathbf{u}}.
\]
With supports remembered, we have 
\[
        |T_3|
        =
        \sum_{I\in\binom{[4]}3}
        \sum_{\mathbf{v}\in[4]^3} t_{I,\mathbf{v}},\quad
        |P_2|
        =
        \sum_{J\in\binom{[4]}2}
        \sum_{\mathbf{u}\in[4]^2} p_{J,\mathbf{u}}.
\]

The remainder of the proof is verified by exact rational Farkas certificates given in Appendix~\ref{app:topstar-spec}. Note that, in view of \eqref{EQ:M=64-c+l}, it suffices for the proof of the lemma to establish
\[
300|T_3|+144|P_2| \le 625(64-c+\ell),
\]
which is equivalent to
\[
625c-625\ell+300|T_3|+144|P_2|\le 40000.
\]
Thus, define the linear objective
\[
\begin{aligned}
        d^TX
        &:=
        625\sum_{x\in [4]^3}y_x
        -
        625\sum_{x\in [4]^3}\sum_{t=1}^3m_{x,t}+300\sum_{I\in\binom{[4]}3}
             \sum_{\mathbf{v}\in[4]^3}t_{I,\mathbf{v}}
        +
        144\sum_{J\in\binom{[4]}2}
             \sum_{\mathbf{u}\in[4]^2}p_{J,\mathbf{u}}.
\end{aligned}
\]
The goal is to show $d^TX\le 40000$.

Recall that $c=\sum_{x\in[4]^3}y_x\le 37$. 
We divide the feasible configurations into branches according to the
level of \(c\). For every \(c\neq23\), one branch indexed by \(c\) is
sufficient. The exceptional case \(c=23\) is refined further according
to the level of \(\ell\). In this case \(\ell\le23+2=25\). Thus the
branch set is
\[
\mathcal B
=
\{c\in\mathbb Z:0\le c\le37,\ c\neq23\}
\cup
\{(23,\ell):\ell\in\mathbb Z,\ 0\le\ell\le25\}.
\]

For a branch \(\beta=c\), the condition system \(A_\beta X\le b_\beta\)
consists of the inequalities of  types {\rm (R1)}--{\rm (R10)} listed in Appendix~\ref{app:topstar-spec},
together with the condition
\[
\sum_{x\in[4]^3}y_x=c.
\]
For a branch \(\beta=(23,\ell)\), we impose both
\[
\sum_{x\in[4]^3}y_x=23
\qquad\text{and}\qquad
\sum_{x\in[4]^3}\sum_{t=1}^3m_{x,t}=\ell.
\]
It is enough to verify the desired objective bound separately for every
\(\beta\in\mathcal B\).

The above is the computer-assisted input in this lemma. By proof-critical files of Appendix \ref{app:topstar-spec}.
We get the following conclusion: For every branch \(\beta\) there exist a positive integer \(D_\beta\), nonnegative row multipliers \(\lambda_\beta\), and nonnegative upper-bound multipliers \(\mu_\beta\)
such that
\[
        A_\beta^T\lambda_\beta+\mu_\beta\ge D_\beta d, 
        \tag{F1}
\]
 and
\[
        b_\beta^T\lambda_\beta+\mathbf 1^T\mu_\beta
        \le 40000D_\beta .
        \tag{F2}
\]
We now explain why \((F1)\) and \((F2)\) imply the desired
bound.  Taking the inner product of (F1) with $X$, we obtain
\begin{equation}\label{EQ:mian-inequ-Farkas-certificat}
         D_\beta d^TX \le \bigl(A_\beta^T\lambda_\beta+\mu_\beta\bigr)^TX
        \le
        \lambda_\beta^T A_\beta X+\mu_\beta^TX.   
\end{equation}

Because \(X\) is feasible for the branch relaxation, we have $A_\beta X\le b_\beta$. Moreover, each coordinate of $X$ does not exceed $1$, and \(\mu_\beta\ge0\). We have $\mu_\beta^TX\le \mu_\beta^T\mathbf 1$. This together with \(\lambda_\beta\ge0\) and \eqref{EQ:mian-inequ-Farkas-certificat} gives 
\begin{align*}
        D_\beta d^TX \le \lambda_\beta^T b_\beta+\mu_\beta^T\mathbf 1  =     b_\beta^T\lambda_\beta+\mathbf 1^T\mu_\beta \le  40000D_\beta,
\end{align*}

where the last inequality is (F2).  Since \(D_\beta>0\), we
obtain
\[
        d^TX\le40000, 
\]
as required.
\end{proof}

\subsection{The no-top-star branch: exact up-set searches}
Recall that \( V_\tau = R_{\F} \cup E_1 \cup E_2 \cup E_3 \cup E_4 \), where for each \( k \in [4] \), $E_k = \{\, e_{k,1} < e_{k,2} < e_{k,3} < e_{k,4} \,\}$.
For a trace \( H \in \F_0 \), its support is 
\[
\operatorname{supp}(H) \coloneqq \{\, k \in [4] : H \cap E_k \neq \emptyset \,\}.
\]
We say that \( H \) is \emph{upper} if it contains none of the elements \( e_{k,1} \) for any \( k \in [4] \). 
A matching in \( \F_0 \) is called \emph{upper} if every trace in it is upper. Let $U=\{2,3,4\}^4$ and $Q_4^U=Q_4\cap U$. 

\begin{lemma}\label{lem:upper-matching-branch}
If $U\setminus Q_4^U$ contains three pairwise disjoint upper quads, then
$T_3=P_2=\emptyset$.
\end{lemma}

\begin{proof}
Let $Q^1,Q^2,Q^3\in U\setminus Q_4^U$ be pairwise disjoint upper
quads.  Suppose, for a contradiction, that a wide pair or triple $H$ is
present.  Choose a column $j\notin\operatorname{supp}(H)$.  By
Lemma~\ref{lem:seeds-r4}, the seed
$S:=R_{\F}\cup\{e_{j,1}\}$ is present and disjoint from $H$.

For $k\in\operatorname{supp}(H)$, let $h_k$ be the level used by $H$ in
column $k$.  For every $i\in[3]$ and $k\in[4]$, define
\[
        \widetilde Q^i(k):=
        \begin{cases}
        1, & \text{if }k\in\operatorname{supp}(H)\text{ and }Q^i(k)=h_k,\\
        Q^i(k), & \text{otherwise.}
        \end{cases}
\]
In each column, the three original upper levels are $2,3,4$.  At most one
of them is lowered, and it is lowered precisely to avoid the level used by
$H$.  No lowering occurs in column $j$, so the modified quads remain
disjoint from the seed.  Consequently, $\widetilde Q^1,\widetilde Q^2,
\widetilde Q^3$ are pairwise disjoint wide quads, disjoint from both $H$ and
$S$, and $\widetilde Q^i\preceq Q^i$ for every $i\in[3]$.  Stability gives
$\widetilde Q^i\in\F_0$.
Therefore,
\[
        H,\quad S,\quad \widetilde Q^1,\quad \widetilde Q^2,
        \quad \widetilde Q^3
\]
are five pairwise disjoint present traces in $V_\tau$, contradicting
Lemma~\ref{lem:local-matching-r4}.
\end{proof}

Hence Lemma~\ref{THM:Lay1-main-result} holds whenever $U\setminus Q_4$ contains three pairwise disjoint upper quads. It therefore remains only to prove the following lemma.

Before that, we say that $T_I$ has the \emph{mixed triple pattern} if
\[
T_I=S^3\cup(L\times S\times S)
\cup(S\times L\times S)\cup(S\times S\times L),
\]
where $S=\{1,2\}$ and $L=\{3,4\}$.
\begin{lemma}\label{cert:no-topstar-main}
Assume that the present upper cube $U\setminus Q_4$ contains no three pairwise disjoint upper quads and that $Q_4^U$ contains no full top-star. Then every genuine no-top-star \(19\)-board satisfies
\[
        300|T_3|+144|P_2|\le 625|Q_4|.
\]
\end{lemma}

\begin{proof}

We claim that every $19$-board in the no-top-star branch satisfies conditions \textup{(N1)}--\textup{(N5)}  listed in Appendix~\ref{app:notopstar-spec}.
Indeed, \textup{(N1)} is the up-set property of \(Q_4\), which follows from the stability. \textup{(N2)} and \textup{(N3)} are exactly the assumptions that the present upper cube has no \(3\)-matching and that \(Q_4^U\) has no full top-star. 
For \textup{(N4)}, Lemmas~\ref{lem:pair-ferrers} and~\ref{lem:triple-ferrers} show that the supported pair and triple trace sets are down-sets and exclude, respectively, the forbidden pair and bad \(3\)-matching configurations.
Finally, completion closure in \textup{(N5)} is Lemma~\ref{fact:closure-general}. 
If a present trace and a disjoint seed from Lemma~\ref{lem:seeds-r4} left three pairwise disjoint present wide quads, these five traces would contradict Lemma~\ref{lem:local-matching-r4}; hence at least one of the three quads must lie in \(Q_4\), which is the required seed-hitting condition.

Conclusion~\ref{Con:no-topstar-upper-blocker} concerns only $U=\{2,3,4\}^4$.  
It is therefore
encoded uniquely by a monotone \(3\times3\times3\) height matrix
\(h\), where
\[
(a,b,c,d)\in U\setminus Q_4\quad\Longleftrightarrow\quad 2\le d < 2+ h[a,b,c].
\]
The no-top-star condition forces the four unit-top points to belong to
$U\setminus Q_4$. 
Moreover, after ordering three disjoint upper quads by their first coordinates, their other three coordinates are permutations of \(2,3,4\).
Thus all possible upper \(3\)-matchings are represented by the \(6^3=216\) triples of permutations. The program enumerates every monotone height matrix satisfying the four forced inclusions and having \(|U\setminus Q_4|\ge 49\), and tests it against all these matching templates. 
The branch-and-bound pruning discards a partial matrix only when the largest possible size of any monotone completion is at most \(48\), so it cannot discard a candidate relevant to the assertion. 
In the certified run, all \(3{,}116{,}823\) candidates of size at least \(49\) contain an upper \(3\)-matching. The program also checks an explicit admissible down-set of size \(48\), showing that the computed maximum is exact.
Consequently, every configuration satisfying \textup{(N1)}--\textup{(N3)} and containing no upper \(3\)-matching has
\(|U\setminus Q_4|\le48\), and hence
\[
|Q_4^U|=|U|-|U\setminus Q_4|\ge3^4-48=81-48=33.
\]
This computation uses only exhaustive integer enumeration and exact bit-mask containment tests; no numerical tolerance or heuristic optimization is involved. 

By Conclusion~\ref{Con:no-topstar-upper-blocker}, we have $|Q_4|\ge 33$, thus we only need to consider $33\le |Q_4|\le 66$. 
Conclusion~\ref{lem:no-topstar-trace-thresholds}, which is obtained from two finite computations, is used to handle the case where \(33\le |Q_4|\le63\). For a fixed pair or triple trace \(H\), the first
computation searches over the \(256\) possible wide quads and determines
a certified threshold \(\theta(H)\) such that the presence of \(H\)
forces \(|Q_4|\ge\theta(H)\). Closure forbids all wide-quad extensions of
\(H\) from \(Q_4\), while every upper matching and every residual
3-matching arising from a disjoint seed becomes a hitting clause
that \(Q_4\) must meet. 
The exact branch-and-bound solver chooses an unhit clause, adds the up-closure of one of its quads, and uses only logically valid unit-propagation and cardinality lower-bound pruning.
Thus exhausting all nodes below a target \(B\) proves that no feasible
up-set of size less than \(B\) exists.

There are only \(10\) pair-level orbits and \(20\) triple-level orbits
under permutations of the selected columns, so \(30\) such threshold
checks suffice. The second computation enumerates all legal down-sets
on one support: \(10\) pair down-sets and \(26{,}893\) triple down-sets
satisfying the obstructions in \textup{(N4)}. For a given \(m\), it
retains only trace levels with \(\theta(H)\le m\) and maximizes the size
of a legal supported down-set. This gives
\[
|P_J|\le p(m),
\qquad
|T_I|\le t(m),
\]
where \(p(m)\) and \(t(m)\) are listed in
Table~\ref{tab:no-topstar-thresholds}. This is a safe relaxation:
individual feasibility is only a necessary condition, and allowing all
individually feasible traces can only increase the resulting maximum.
There are 6 pair supports and 4 triple supports, so
\[
|P_2|\le6p(m),
\qquad
|T_3|\le4t(m).
\]
Substituting the eight ranges in
Table~\ref{tab:no-topstar-thresholds} gives that
\[
300\cdot4t(m)+144\cdot6p(m)\le625m.
\]
Therefore \(300|T_3|+144|P_2|\le625m\) throughout this range.

It remains to consider \(64\le m\le66\). Call a pair support \(J\)
extremal if \(|P_J|=4\), and a triple support \(I\) extremal if
\(|T_I|=32\); these are the maximum levels allowed by
Lemmas~\ref{lem:pair-ferrers} and~\ref{lem:triple-ferrers}. If no
support is extremal, then
\[
|P_2|\le6\cdot3=18,
\qquad
|T_3|\le4\cdot31=124,
\]
and hence
\[
300|T_3|+144|P_2|
\le300\cdot124+144\cdot18
=39792<625\cdot64\le625m.
\]
We are done. 

Suppose  that some support is extremal.
Conclusion~\ref{lem:no-topstar-critical-range} analyzes the equality patterns from the two Ferrers bounds by $9$ exact searches. 
Up to coordinate permutation, the possible patterns are one pair square, three triple slabs, and one mixed triple diagram. The pattern solver inserts every trace in the chosen diagram, imposes its closure and residual seed-hitting clauses, and then exhausts all missing up-sets below the stated target. 
The mixed triple pattern has certified minimum \(|Q_4|=80\), and hence cannot occur when \(|Q_4|\le66\). 
For each of the remaining patterns the solver also tests the negation of the required rectangle. 
Since \(Q_4\) is an up-set, saying that \(R_{ij}\nsubseteq Q_4\) is equivalent to forbidding its minimal point.
The resulting search has no feasible solution of size at most \(66\); hence one of the high-level rectangles
\[
R_{ij}:=\{q\in[4]^4:q_i\ge 3,\ q_j\ge 3\}
\]
must be contained in \(Q_4\).

Each such rectangle has \(2^2\cdot4^2=64\) points. Therefore an up-set
\(Q_4\supseteq R_{ij}\) with \(|Q_4|\le66\) differs from the rectangle by at
most two points, together with their required up-closures. The second
stage of the audit enumerates all such extensions for the six choices
of \(i,j\), removes duplicates, and retains those satisfying the
no-top-star and no-upper-matching conditions; exactly \(72\) cases
remain. For each fixed \(Q_4\), it checks closure and every residual seed
cut for all pair and triple levels, and then maximizes independently on
each support over the \(10\) legal pair down-sets and \(26{,}893\) legal
triple down-sets. Independence can only enlarge the set of globally
compatible traces, so the resulting sum is a rigorous upper bound. Its
worst case is
\[
|Q_4|=64,\qquad |T_3|=96,\qquad |P_2|=4,
\]
for which
\[
300|T_3|+144|P_2|-625|Q_4|
=-10624<0.
\]
Thus the required inequality also holds whenever an extremal support
is present. All possible values of \(m\) have now been covered.
This completes the proof. 
\end{proof}

\subsection{Proof of Theorem \ref{THM:Lay1-main-result}}

Now we prove Theorem \ref{THM:Lay1-main-result}. If $|Q_4|\ge67$, then we are done by \eqref{eq:T3-P2-M}. If $Q_4$ contains a full top-star, Lemma~\ref{lem:topstar} applies. Suppose $Q_4$ contains no full top-star. 
In this case, if the present upper wide quad contains an $3$-matching, then Lemma~\ref{lem:upper-matching-branch} gives $T_3=P_2=\emptyset$. So Theorem \ref{THM:Lay1-main-result}  is  obviously. Otherwise, Lemma~\ref{cert:no-topstar-main} applies.  

\section{Layer $2$: the $15$-board inequality}\label{sec:layer2}

On a $19$-board $V_\tau$, let $T_2$ be present triples of spread $2$, $P_1$ be present pairs of spread $1$, and $Q_3$ be missing quads of spread $3$ in $V_\tau$. The following is the main result in this section. 
\begin{theorem}\label{cor:layer2}
On every $19$-board, 
\[
\frac{24}{25}|T_2|+\frac{432}{625}|P_1|\le |Q_3|.
\]
\end{theorem}
For a three-column subset $\sigma\subseteq\tau$, define the $15$-board $V_\sigma\coloneqq R_{\F}\cup\bigcup_{i\in\sigma}E_i$.
Let $T_\sigma$ be present triples of spread $2$ in this $15$-board, $P_\sigma$ be present pairs of spread $1$, and $Q_\sigma$ be missing quads of spread $3$.  
\begin{theorem}\label{prop:15-board}
Every genuine \(15\)-board arising as a three-column subboard $\sigma$ of a genuine \(19\)-board satisfies 
\begin{equation}\label{eq:15-target}
 300|T_\sigma|+144|P_\sigma|\le 625|Q_\sigma|.
\end{equation}
\end{theorem}

\begin{proof}[Proof of Theorem \ref{cor:layer2} assuming Theorem \ref{prop:15-board}]
Fix a $19$-board $V_\tau$, and a three-column subset $\sigma\subseteq\tau$. 
By Theorem~\ref{prop:15-board}, \begin{equation*}\label{eq:15-target}
        \frac{12}{25}|T_\sigma|+\frac{144}{625}|P_\sigma|\le |Q_\sigma|. 
\end{equation*}
Summing over the four $\sigma$ inside $\tau$ completes the proof. This is because a spread-$2$ triple is seen in two $15$-boards, a spread-1 pair is seen in three, and a spread-$3$ missing quad is seen in exactly one.
\end{proof}

Now we prove Theorem \ref{prop:15-board}
\begin{proof}[Proof of Theorem \ref{prop:15-board}]
Without loss of generality, we write the $15$-board as $V_{\sigma}=R_{\F}\cup E_1\cup E_2\cup E_3$.
$E_4$ is the fourth selected column of the $19$-board, the one not contained in this $15$-board.
For a  quad \(H_Q\subseteq V_{\sigma}\) of spread \(3\), let \(m_{H_Q}\) be its missing indicator, that is, $m_{H_Q}=1$ if and only if $H_Q\notin \F^0$. 
For a spread-\(2\) triple or a  spread-\(1\)  pair \(H\subseteq V_{\sigma}\), let \(x_H\) be its present indicator.
 
First, suppose that \(H\) is a spread-\(2\) triple or a spread-\(1\) pair, and suppose that \(H_Q^1, H_Q^2, H_Q^3\) are three pairwise disjoint spread-\(3\) quads, all disjoint from \(H\).  Then
\begin{align}\label{eq:x<mQ123}
    x_H\le m_{H_Q^1}+m_{H_Q^2}+m_{H_Q^3}.
\end{align}
Indeed, if \(H, H_Q^1, H_Q^2, H_Q^3\) were all present, then together with the outside full column \(E_4\) they would form five pairwise disjoint present traces in the ambient $19$-board, contradicting Lemma~\ref{lem:local-matching-r4}.

Second, by the same reason, suppose that \(S\) is one of the seed from $E_1, E_2$ or $E_3$, disjoint from \(H\), and that \(H_Q^1, H_Q^2\) are two pairwise disjoint spread-\(3\) quads, disjoint from both \(H\) and \(S\).  Then
\begin{align}\label{eq:x<mQ12}
    x_H\le m_{H_Q^1}+m_{H_Q^2}.
\end{align}

We refer to each such configuration of five pairwise disjoint traces, with $H$ as the distinguished trace, as an \emph{$H$-row}.
Let $T^1, T^0, P^1, P^0$ denote respectively the following four classes of possible traces:
\[
\begin{array}{ll}
T^1:\text{spread-\(2\) triples containing one vertex of }R_{\F},\\
T^0:\text{spread-\(2\) triples containing no vertex of }R_{\F},\\
P^1:\text{spread-\(1\) pairs containing one vertex of }R_{\F},\\
P^0:\text{spread-\(1\) pairs containing no vertex of }R_{\F}.
\end{array}
\]
Then
\[
|T_\sigma|=\sum_{H\in T^1\cup T^0}x_H,
\qquad
|P_\sigma|=\sum_{H\in P^1\cup P^0}x_H.
\]
We now take an explicit nonnegative weighted sum of inequalities of the forms
\eqref{eq:x<mQ123} and \eqref{eq:x<mQ12}.  More precisely, for each possible
trace $H$, we enumerate suitable decompositions of the remaining vertices into
pairwise disjoint spread-$3$ quads, with a seed included when necessary.  Each
such decomposition gives one valid inequality, which we regard as a row, and
we assign a nonnegative weight to every row.  The weights are chosen so that,
after summing all rows, every spread-$2$ triple has total left-hand coefficient
$300$ and every spread-$1$ pair has total left-hand coefficient $144$.  We then
count the total coefficient, or load, received by each variable $m_Q$ on the
right-hand side and verify that it is at most $625$.

\medskip
\noindent
\textbf{1. For \(T^1\).}
Let \(H\in T^1\).  Suppose, without loss of generality, that \(H\) contains one vertex of \(R_{\F}\), one vertex of \(E_1\), and one vertex of \(E_2\).
Then \(V_{\sigma}\setminus H\) has
\[
2 \text{ vertices in }R_{\F},\qquad
3 \text{ vertices in }E_1,\qquad
3 \text{ vertices in }E_2,\qquad
4 \text{ vertices in }E_3.
\]
Later on, we will use the form $R_{\F}^2E_1^3E_2^3E_3^4$ to express this meaning. We take all unordered decompositions
\[
V_{\sigma}\setminus H=H_Q^1\cup H_Q^2\cup H_Q^3
\]
into three pairwise disjoint spread-\(3\) quads. 
There are \(432\) such decompositions.  Here is the count.  
Since \(E_3\) has four remaining vertices, exactly one of the three quads must contain two \(E_3\)-vertices, one $E_1$-vertex and $E_2$-vertex. 
The remaining two \(E_3\)-vertices distinguish the two remaining quads. We then match the two remaining \(E_1\)-vertices, the two remaining \(E_2\)-vertices, and the two remaining \(R_{\F}\)-vertices, to these two quads.  
Hence the number of $T^1$-rows is
\[
\binom42\cdot 3\cdot 3\cdot 2!\cdot 2!\cdot 2!=432.
\]
Each of the $432$ decompositions gives one inequality of the form
\eqref{eq:x<mQ123}.  Summing these inequalities gives
\[
 432x_H\le
 \sum_{H_Q^1\cup H_Q^2\cup H_Q^3=V\setminus H}
 (m_{H_Q^1}+m_{H_Q^2}+m_{H_Q^3}).
\]

We now give each of these $432$ inequalities weight $\frac{25}{36}$. Then every \(H\in T^1\) receives total left-hand coefficient \(300\).
Note that in the following discussion on $m_Q$, the $m_Q$ associated with $T^1$ is similarly multiplied by this weight. The same applies to the rest with their respective weights, a convention we will not emphasize further.

\medskip
\noindent
\textbf{2. For \( T^0\).}
Let \(H\in T^0\). 
Without loss of generality, we assume $|H \cap E_1|=2$, $|H\cap E_2|=1$.

For this fixed $H$, let \(\mathcal O(H)\) be the family of inequality rows
obtained by choosing a seed in the column omitted by $H$ and then decomposing
the remaining vertices into two quads.  
Here the omitted column is $E_3$, so the seed is $R_{\F}\cup\{e_{3,j}\}$, with two possible choices by Lemma \ref{lem:seeds-r4}.  
After deleting $H$ and the chosen seed, the remaining type is $E_1^2E_2^3E_3^3$.  
Every unordered decomposition of these eight vertices into two pairwise disjoint spread-\(3\) quads gives one row of the form \eqref{eq:x<mQ12}, and \(\mathcal O(H)\) consists of all such rows.
The two $A$-vertices label the two quads, so we get $2\binom32\binom31=18$ decompositions for each such seed. Thus,  
\begin{align}\label{|O(H)|=36}
    |\mathcal O(H)|=36.
\end{align}
The second family, denoted \(\mathcal S(H)\), consists of the analogous rows
obtained by choosing the seed in the singleton column $E_2$, namely
$R_{\F}\cup\{e_{2,j}\}$. If $H\cap E_2\in\{e_{2,3}, e_{2,4}\}$, then both seeds in \(E_2\) are available.  If $H\cap E_2\in\{e_{2,1}, e_{2,2}\}$, then only the other seed is available.
For a fixed available seed in \(E_2\), the remaining type is $E_1^2E_2^2E_3^4$.
The two \(E_1\)-vertices label the two quads, this gives
\[
2\binom42=12
\]
decompositions for each fixed seed. Hence
\begin{align}\label{|S(H)|=24/12}
    |\mathcal S(H)|=
\begin{cases}
24,& H\cap E_2\in \{e_{2,3}, e_{2,4}\},\\
12,& H\cap E_2\in \{e_{2,1}, e_{2,2}\}.
\end{cases}
\end{align}

We now assign weights.  If $H\cap E_2\in \{e_{2,3}, e_{2,4}\}$, we use only the \(\mathcal O(H)\)-rows, each with weight $\frac{25}{3}$. 
Then, by \eqref{|O(H)|=36},
\[
36\cdot\frac{25}{3}=300.
\]

If $H\cap E_2\in \{e_{2,1}, e_{2,2}\}$, we use the \(\mathcal O(H)\)-rows with weight $\frac{23}{12}$, and the \(\mathcal S(H)\)-rows with weight $\frac{77}{4}$.
Then, by \eqref{|O(H)|=36} and \eqref{|S(H)|=24/12},
\[
36\cdot\frac{23}{12}
+
12\cdot\frac{77}{4}
=
69+231
=
300.
\]
Thus every \(H\in T^0\) receives total left-hand coefficient
\(300\).

\medskip
\noindent
\textbf{3. For \(P^1\).}
Let \(H\in P^1\).  Suppose, for example, that \(H\) contains one vertex of \(R_{\F}\) and one vertex of \(E_1\).  We choose one unused vertex $u\in V_{\sigma}\setminus H$ and then decompose
\[
V_{\sigma}\setminus(H\cup\{u\})
=
H_Q^1\cup H_Q^2\cup H_Q^3
\]
into three pairwise disjoint spread-\(3\) quads.
The unused vertex \(u\) cannot lie in \(E_1\), otherwise only two \(E_1\)-vertices would remain, and three spread-\(3\) quads would not all be able to meet \(E_1\).
Hence, there are two cases.

If \(u\in R_{\F}\), then there are \(2\) choices for \(u\).  The remaining vertices are $R_{\F}^1E_1^3E_2^4E_3^4$.
The three \(E_1\)-vertices label the three quads. One quad must contain two \(E_2\)-vertices, one must contain two \(E_3\)-vertices, and one must contain the single vertex of \(R_{\F}\). 
Thus there are $3!\binom42\binom42\cdot 2\cdot 2=864$ decompositions for each such \(u\).

If $u\in E_2 \cup E_3$, then there are \(8\) choices for \(u\). Without loss of generality, the remaining vertices are $R_{\F}^2E_1^3E_2^3E_3^4$.
Hence there are \(\binom{4}{2}\cdot3\cdot3\cdot2\cdot2\cdot2=432\) decompositions for each such \(u\).

Therefore the number of rows for a fixed \(H\in P^1\) is
\[
2\cdot 864+8\cdot 432=5184.
\]
For every one of these decompositions we use \eqref{eq:x<mQ123} with weight $\frac1{36}$. Then every \(H\in P^1\) receives left-hand coefficient \(144\).

\medskip
\noindent
\textbf{4. For \(P^0\).}
Let \(H\in P^0\).  Suppose, without loss of generality, that \(H\subseteq E_1\). Choose a seed in one of the two other columns \(E_2, E_3\).  Thus the seed is one of
\[
R_{\F}\cup\{e_{2,1}\},\quad R_{\F}\cup\{e_{2,2}\},\quad
R_{\F}\cup\{e_{3,1}\},\quad R_{\F}\cup\{e_{3,2}\}.
\]
After choosing the seed, choose one unused vertex and decompose the remaining eight vertices into two spread-\(3\) quads.  We use the seed version \eqref{eq:x<mQ12}.

For a fixed seed, say \(R_{\F}\cup\{e_{2,i}\}\).
If the unused vertex lies in \(E_2\), there are \(3\) choices for it, and the remaining type is \(E_1^2E_2^2E_3^4\).  The two \(E_1\)-vertices label the two quads; after choosing the matching \(B\)-vertex and the two \(E_3\)-vertices for the first \(E_1\)-labelled quad, we get $2\binom42=12$ decompositions for each such unused vertex.

If the unused vertex lies in \(E_3\), there are \(4\) choices for it,
and the remaining type is \(E_1^2E_2^3E_3^3\).  Again the two \(E_1\)-vertices
label the two quads. We get $2\binom32\binom31=18$ decompositions for each such unused vertex.

Thus, for each fixed seed, the number of decompositions is $3\cdot 12+4\cdot 18=108$.
There are \(4\) possible seeds, so a fixed \(H\in P^0\) generates $4\cdot108=432$ rows.  We give each such row weight $\frac13$. Therefore every \(H\in P_0\) receives total coefficient $144$.

Combining the four classes, the left-hand side of the total weighted
sum is exactly
\[
300|T_\sigma|+144|P_\sigma|.  
\]

For a spread-$3$ quad $H_Q$, define its \emph{load}, to be the total coefficient with which $m_{H_Q}$ appears on the right-hand side of the weighted sum of the selected inequalities. 
It remains to prove that no spread-\(3\) missing quad \(H_Q\) receives load greater than \(625\).  

\medskip
\noindent
\textbf{Case 1. Load when \(H_Q\cap R_{\F} \neq\emptyset\).}

Note that \(H_Q\) contains exactly one vertex of \(R_{\F}\) and one vertex in each of \(E_1, E_2, E_3\).
Only the \( T^1\)-rows and the \(P^1\)-rows can contain such a \(H_Q\).  

First count the \(T^1\)-rows containing this fixed \(H_Q\).
Choose the two selected columns used by \(H\); there are \(3\) choices.
For a fixed choice, say \(E_1, E_2\), the trace \(H\) must choose one of the two $R_{\F}$-vertices not in \(H_Q\), one of the three \(E_1\)-vertices not in \(H_Q\), and one of the three \(E_2\)-vertices not in \(H_Q\).  
This gives $2\cdot3\cdot3=18$ choices for \(H\). After deleting \(H_Q\cup H\), the remaining vertices have type $R_{\F}^1E_1^2 E_2^2 E_3^3$.
They can be decomposed into two spread-\(3\) quads in $2\cdot2\cdot3=12$ ways.
Thus the number of \(T^1\)-rows containing \(H_Q\) is $3\cdot18\cdot12=648$. 

Since each has weight \(25/36\), the \(T^1\)-load is
\begin{equation}\label{EQ:load-Q-T1}
   648\cdot\frac{25}{36}=450.  
\end{equation}

Next count the \(P^1\)-rows containing \(H_Q\). Suppose first that \(H\) uses column \(E_1\). There are $2\cdot3=6$ choices for \(H\), since it must choose one remaining vertex and one \(E_1\)-vertex outside \(H_Q\).  
After deleting \(H_Q\cup H\), the remaining type is $R_{\F}^1E_1^2 E_2^3 E_3^3$.
If the unused vertex is the remaining \(R_{\F}\)-vertex, then the remaining type is \(E_1^2 E_2^3 E_3^3\), giving \(18\) decompositions.  
If the unused vertex lies in \(E_2\cup E_3\), then there are \(6\) choices for it, and each choice gives \(12\) decompositions.  Hence, for each fixed \(H\), the number of this rows is $18+6\cdot12=90$.
Thus column \(E_1\) contributes \(6\cdot90=540\) pairs.  The same count
holds for columns \(E_2\) and \(E_3\), so the total number of \( P^1\)-rows containing \(H_Q\) is $3\cdot540=1620$.
Since each has weight \(1/36\), the \( P^1\)-load is
\begin{equation}\label{EQ:load-Q-P1}
1620\cdot\frac1{36}=45. 
\end{equation}

Therefore, combining \eqref{EQ:load-Q-T1} and \eqref{EQ:load-Q-P1}, if \(Q\cap R_{\F}\neq\emptyset\), the coefficient of $m_Q$ is $450+45=495\le 625$.

\medskip
\noindent
\textbf{Case 2. Load when \(H_Q \cap R_{\F}=\emptyset\).}

Without loss of generality, assume $|Q\cap E_1|=2$ and $|Q\cap E_2|=|Q\cap E_3|=1$. 
For the \(T^1\)-rows, $H$ must use exactly one vertex from each of $R_{\F}$, $E_2$, and $E_3$. There are $3\cdot3\cdot3=27$ choices for \(H\). 
After deleting \(H_Q\cup H\), the remaining type is $R_{\F}^2 E_1^2 E_2^2 E_3^2$,
which can be decomposed into two spread-\(3\) quads in \(2^3=8\) ways.  
Thus the number of \(T^1\)-rows containing \(H_Q\) is $27\cdot8=216$, and their contribution to coefficient of $m_{H_Q}$ is
\begin{align}\label{eq:150}
    216\cdot\frac{25}{36}=150.
\end{align}

For the \(P^1\)-rows, \(H\) must use exactly one vertex from $R_{\F}$ and $E_2$ (or $E_3$). Suppose \(H\) uses \(E_2\). There are $3\cdot3=9$ choices for \(H\). After deleting \(H_Q\cup H\), the remaining type is $R_{\F}^2 E_1^2 E_2^2 E_3^3$.
If the unused vertex lies in \(R_{\F}\), there are \(2\) choices and each gives \(12\) decompositions.  If it lies in \(E_3\), there are \(3\) choices and each gives \(8\) decompositions.  Thus, for each \(H\), there are
\[
2\cdot12+3\cdot8=48
\]
rows.  
Hence the choice \(H\) in column \(E_2\) contributes \(9\cdot48=432\) rows.  The same holds with \(E_3\).  
Thus the total number of \(P^1\)-pairs containing \(H_Q\) is $2\cdot432=864$, and their contribution to coefficient of $m_{H_Q}$ is
\begin{align}\label{eq:24}
    864\cdot\frac1{36}=24. 
\end{align}
Let $a:=|Q\cap \{e_{1,1}, e_{1,2}\}|\in\{0,1,2\}$,
and $\mathbf{y}, \mathbf{z}$ be indicate function whether $H_Q\cap E_2\in \{1, 2\}$ and $Q\cap E_3\in \{1, 2\}$, respectively.  
Put
\[
k:=\mathbf{y}+\mathbf{z}\in\{0,1,2\}.
\]
For the \(P^0\)-rows, \(H\) must be a pair inside $E_2$ or $E_3$.  
Suppose first that \(H\subseteq E_2\), then \(H\) can be chosen in \(\binom32=3\) ways. 
If the seed $S$ is in \(E_3\),  the seed can be chosen in \(2-\mathbf{z}\) ways, and after \(H_Q,H,S\) are fixed there are \(4\) choices to spread-$3$ quads by remaining vertices. The number of such rows is
\[
3\cdot(2-\mathbf{z})\cdot4=12(2-\mathbf{z}).
\]

If the seed is in \(E_1\), the seed can be chosen in \(2-a\) ways, and after \(H_Q,H,S\) are fixed there are \(3\) choices to spread-$3$ quads by remaining vertices. The number of such rows is
\[
3\cdot(2-a)\cdot3=9(2-a).
\]
The same argument with \(E_2, E_3\) interchanged gives $12(2-\mathbf{y})+9(2-a)$ rows. Therefore the total number of \(P^0\)-rows containing \(Q\) is
\[
12(2-\mathbf{z})+12(2-\mathbf{y})+18(2-a)
=
84-18a-12k.
\]
Since each such row has weight $\frac{1}{3}$, the contribution to coefficient of $m_{H_Q}$ is
\begin{align}\label{eq:28-6r-4k}
    \frac13(84-18a-12k)=28-6a-4k. 
\end{align}
It remains to count the \(T^0\)-rows containing \(H_Q\).  
There is no unused vertex in these rows: once \(H\), the seed \(S\), and \(H_Q\) are fixed, the other variable quad is forced to be $H_Q'=V\setminus(H_Q\cup H\cup S)$.
Thus we only need to count the choices for which this forced \(H_Q'\)
is still spread-\(3\).

For the \(\mathcal O\)-rows, the seed lies in the column disjoint from $H$.  
A direct column-by-column count gives:
\begin{align}\label{eq:type-of-O-row}
    \begin{array}{c|c}
\text{type of \(\mathcal O\)-row} & \text{number of rows containing \(H_Q\)}\\ \hline
\text{singleton of \(H\) is $e_{i,3}$ or $e_{i,4}$}
&
12+6a+6k-6ak
\\[2pt]
\text{singleton of \(H\) is $e_{i,1}$ or $e_{i,2}$}
&
6(2-a)(4-k).
\end{array}
\end{align}

For completeness, we indicate the count.  If the column disjoint from $H$ is \(E_1\), the seed has \(2-a\) choices. For the case of $|E_2\cap H|=2$ and $|E_3\cap H|=1$, the number of choices of $E_3\cap H\in \{e_{3,3}, e_{3,4}\}$ is $1+\mathbf{z}$,  $E_3\cap H\in \{e_{3,1}, e_{3,2}\}$ is $2-\mathbf{z}$. 
The case with $|E_2\cap H|=1$ and $|E_3\cap H|=2$ is the same.

The number of \(\mathcal O\)-rows where singleton of \(H\) is $e_{i,3}$ or $e_{i,4}$ is
\[
(2-a)\cdot\binom32\cdot(1+\mathbf{z})+(2-a)\cdot\binom32\cdot(1+\mathbf{y})=3(2-a)(2+k),
\]
singleton of \(H\) is $e_{i,1}$ or $e_{i,2}$ is
\[
(2-a)\cdot\binom32\cdot(2-\mathbf{z})+(2-a)\cdot\binom32\cdot(2-\mathbf{y})=3(2-a)(4-k).
\]  
If the column disjoint from $H$ is \(E_2\) or \(E_3\), it is obvious that $|H\cap E_1|=1$. 
When $|H\cap E_3|=\emptyset$, the seed have $2-\mathbf{z}$ choices and $|H\cap E_2|=\emptyset$ is the same.
Hence, the number of \(\mathcal O\)-rows where singleton of \(H\) is $e_{i,3}$ or $e_{i,4}$ is
\[
a\cdot\binom32\cdot(2-\mathbf{z})+a\cdot\binom32\cdot(2-\mathbf{y})=3a(4-k),
\]
singleton of \(H\) is $e_{i,1}$ or $e_{i,2}$ is
\[
(2-a)\cdot\binom32\cdot(2-\mathbf{z})+(2-a)\cdot\binom32\cdot(2-\mathbf{y})=3(2-a)(4-k).
\]    
Adding the two contributions $3(2-a)(2+k)+3a(4-k)$ and $3(2-a)(4-k)+3(2-a)(4-k)$ gives exactly \eqref{eq:type-of-O-row}.

For the \(\mathcal S\)-rows, the seed lies in the singleton column of \(H\).  
In this case \(H\) must use the two singleton columns \(E_2, E_3\) and omit \(E_1\).  
If the singleton of \(H\) is $e_{i,3}$ or $e_{i,4}$, so there are
\[
2\cdot 2 \cdot \binom{3}{2}=12
\]
such rows. 
If the singleton of \(H\) is $e_{i,1}$ or $e_{i,2}$, the number of \(\mathcal S\)-rows is
\begin{align}\label{eq:low-S-rows}
    (1-\mathbf{y}) \binom21 \cdot \binom{3}{2}+(1-\mathbf{z})\binom21\cdot \binom{3}{2}=6(2-k). 
\end{align}

Recall the weights for \(T^0\)-rows:
\[
\begin{array}{c|cc}
 & \mathcal O\text{-row weight} & \mathcal S\text{-row weight}\\ \hline
\text{singleton of \(H\) is $e_{i,3}$ or $e_{i,4}$}
& \dfrac{25}{3} & 0\\[4pt]
\text{singleton of \(H\) is $e_{i,1}$ or $e_{i,2}$}
& \dfrac{23}{12} & \dfrac{77}{4}.
\end{array}
\]
Using \eqref{eq:type-of-O-row} and \eqref{eq:low-S-rows}, the contribution to coefficient of $m_{H_Q}$ is
\begin{align}\label{eq:423+4r-}
\frac{25}{3}(12+6a+6k-6ak)
+\frac{23}{12}\cdot 6(2-a)(4-k)
+\frac{77}{4}\cdot 6(2-k)
=
423+4a-\frac{177+77a}{2}\,k.
\end{align}
Combining \eqref{eq:150}, \eqref{eq:24}, \eqref{eq:28-6r-4k}, and \eqref{eq:423+4r-}, we get the coefficient of $m_Q$ is
\begin{align*}
150+24+(28-6a-4k)+\left(423+4a-\frac{177+77a}{2}k\right)
&=625-2a-\frac{185+77a}{2}k
\le 625,
\end{align*}
as \(a,k\ge0\).  Thus every spread-\(3\) quad \(H_Q\) receives load at most \(625\).

Therefore the weighted sum of all valid remaining inequalities gives
\[
300|T_\sigma|+144|P_\sigma|
\le
625\sum_{H_Q} m_{H_Q}
=
625|Q_\sigma|, 
\]
which completes the proof. 
\end{proof}

\section{Layer $3$: the $11$-board inequality}\label{sec:layer3}

In this section, we prove the following. 
\begin{theorem}\label{cor:layer3}
On every $19$-board, 
\[
\frac{36}{25}|T_1|\le |Q_2|, 
\]
where $T_1$ is the present triples of spread $1$
and $Q_2$ is the missing quads of spread $2$ in $V_{\tau}$.
\end{theorem}
For a two-column subset $\iota \subseteq \tau$, define the $11$-board $V_{\iota}:=R_{\F}\cup E_i\cup E_j$.
Let $T_{\iota}$ be the present triples of spread $1$ in this $11$-board
and $Q_{\iota}$ be the missing quads of spread $2$. 
\begin{theorem}\label{prop:11-board}
For every genuine $11$-board, we have 
\begin{equation}\label{eq:11-target}
        \frac{12}{25}|T_{\iota}|\le |Q_{\iota}|.
\end{equation}
\end{theorem}

\begin{proof}[Proof Theorem \ref{cor:layer3} through Theorem \ref{prop:11-board}]
Sum Theorem~\ref{prop:11-board} over the six two-column subboards inside the $19$-board.  
A spread-$1$ triple is seen in three $11$-boards, and a spread-$2$ missing quad is seen in exactly one.
\end{proof}

\begin{proof}[Proof of Theorem \ref{prop:11-board}]
Without loss of generality, write the $11$-board as $V_{\iota }=R_{\F}\cup E_1\cup E_2$.
For a spread-$2$ quad \( H_Q\subseteq V_{\iota}\), let \(m_{H_Q}\) be its missing indicator. For a spread-$1$ triple \(H\subseteq V_{\iota}\), let \(x_H\) be its present indicator.
We now bound the spread-$1$ triples. 

First consider triples \(H\) which are contained entirely in one selected column.  Suppose \(H\subseteq E_1\), \(|H|=3\).  If \(e_{1,1}\notin H\), then
\[
H, R_{\F}\cup\{e_{1,1}\},\ E_2,\ E_3,\ E_4
\]
are five pairwise disjoint present traces, contradicting Lemma~\ref{lem:local-matching-r4}.  
Similarly, it is impossible that \(e_{1,2}\notin H\).
Thus a present column-only triple in \(E_1\) must be one of $E_1\setminus\{e_{1,3}\}, E_1\setminus\{e_{1,4}\}$.
For \(v\in\{3,4\}\) and \(u\in\{1,2\}\), put
\[
Q^{E_1}_{v,u}:=\{e_{1,v}\}\cup (E_2\setminus\{e_{2,u}\}).
\]
If \(H=E_1\setminus\{e_{1,v}\}\) is present and \(Q^{E_1}_{v,u}\) is also
present, then
\[
H,\quad Q^{E_1}_{v,u},\quad  R_{\F}\cup\{e_{2,u}\},\quad E_3,\quad E_4
\]
are five pairwise disjoint present traces, again impossible.  
Hence
\[
x_{E_1\setminus\{e_{1,v}\}}\le m_{Q^{E_1}_{v,u}}
\qquad (v=3,4;\ u=1,2).
\]
Multiplying the two inequalities for \(u=1,2\) by \(150\) and summing
gives
\begin{align}\label{eq:A-and-150m}
    300x_{E_1\setminus\{e_{1,v}\}} \le 150m_{Q^{E_1}_{v,1}}+150m_{Q^{E_1}_{v,2}}.
\end{align}
The same argument with \(E_1\) and \(E_2\) interchanged gives, for
\(v\in\{3,4\}\),
\begin{align}\label{eq:B-and-150m}
    300x_{E_2\setminus\{e_{2,v}\}}\le 150m_{Q^{E_2}_{v,1}}+150m_{Q^{E_2}_{v,2}}.
\end{align}
where
\[
Q^{E_2}_{v,u}:=\{e_{2,v}\}\cup(E_1\setminus\{e_{1,u}\}).
\]
It remains to treat spread-$1$ triples which meet \( R_{\F}\). 
Let \(H\) be such a triple.  
Let \(\Pi(H)\) be the set of unordered decompositions
\[
V_{\iota}\setminus H=H_Q\cup H_Q'
\]
such that both \(H_Q\) and \(H_Q'\) are two disjoint spread-$2$ quads.  
For every \(\{H_Q, H_Q'\}\in\Pi(H)\), if \(H, H_Q, H_Q'\) were all present, then $H,\ H_Q,\ H_Q',\ E_3,\ E_4$ would be five pairwise disjoint present traces.  
Therefore
\begin{align}\label{eq:xH<mQ+mQ'}
    x_H\le m_{H_Q}+m_{H_{Q'}} \qquad \text{for every } \{H_Q, H_Q'\}\in\Pi(H).
\end{align}
We need only count how many such decompositions there are. 
By symmetry, it suffices to assume \(H\subseteq R_{\F}\cup E_1\).

If \(|H\cap  R_{\F}|=1\), then \(V_{\iota}\setminus H\) has type $R_{\F}^2E_1^2E_2^4$.
In a valid decomposition, each spread-$2$ quads must contain one of the two remaining \(E_1\)-vertices.  
Hence the number of unordered decompositions is
\[
\frac12\binom21
\left[
\binom41\binom22+\binom42\binom21+\binom43\binom20
\right]
=20.
\]
Multiplying each inequality in \eqref{eq:xH<mQ+mQ'} by \(15\) and summing over
\(\Pi(H)\), we get
\begin{align}\label{eq:300xH<15sumPi(H)}
    300x_H\le 15\sum_{\{H_Q, H_Q'\}\in\Pi(H)}(m_{H_Q}+m_{H_{Q'}}). 
\end{align}
If \(|H\cap R_{\F}|=2\), then \(V_{\iota}\setminus H\) has type $R_{\F}^1E_1^3 E_2^4$
The number of unordered decompositions is
\[
\frac12\left[
\binom31\left(\binom42\binom11+\binom43\binom10\right)
+
\binom32\left(\binom41\binom11+\binom42\binom10\right)
\right]
=30.
\]
Multiplying each inequality in \eqref{eq:xH<mQ+mQ'} by \(10\) and summing over
\(\Pi(H)\), we get
\begin{align}\label{eq:300xH<10sumPi(H)}
    300x_H\le 10\sum_{\{H_Q, H_Q'\}\in\Pi(H)}(m_{H_Q}+m_{H_{Q'}}). 
\end{align}

We now sum the inequalities \eqref{eq:A-and-150m}, \eqref{eq:B-and-150m}, \eqref{eq:300xH<15sumPi(H)}, and \eqref{eq:300xH<10sumPi(H)} over all spread-1 triples \(H\). 
We will now prove that each spread-$2$ quad $H_Q$ receives a total load of at most $375$.
For the triples meeting \(R_{\F}\), let
\[
a=|H_Q\cap E_1|,\qquad b=|H_Q\cap E_2|,  \qquad c=|H_Q\cap  R_{\F}|.
\]
This kind of triple load on \(H_Q\) is
\[
L_{\mathrm{mix}}(H_Q)=\sum_{h=1}^{2}w_h\binom{3-c}{h}
\left[\binom{4-a}{3-h}{\bf 1}_{h\ge a}+\binom{4-b}{3-h}{\bf 1}_{h\ge b}\right],
\]
where \(w_1=15\) and \(w_2=10\).  This formula simply chooses the \(h\) $R_{\F}$-vertices and the \(3-h\) vertices in the same selected column as \(H\), and the condition \(h\ge a\) or \(h\ge b\) ensures that the complementary quad still meets that selected column.
Evaluating this elementary expression for the six possible types of
\((c,a,b)\) gives
\[
\begin{array}{c|cccccc}
(c,a,b)     &(0,1,3)&(0,2,2)&(0,3,1)&(1,1,2)&(1,2,1)&(2,1,1)\\ \hline
L_{\mathrm{mix}}(Q)&225&120&225&140&140&90 .
\end{array}
\]

The column-only inequalities \eqref{eq:A-and-150m} and \eqref{eq:B-and-150m} add load only to
quads of the special forms
\[
\{e_{1,v}\}\cup(E_2\setminus\{e_{2,u}\})
\quad\text{or}\quad
(E_1\setminus\{e_{1,u}\})\cup\{e_{2,v}\},
\qquad v\in\{3,4\},\ u\in\{1,2\}.
\]
Each spread-$2$ quad has at most one such representation, so the
column-only part contributes at most \(150\) to any fixed quad.
Consequently the total load on every spread-2 quad is at most $225+150=375.$
Therefore,
\[
300|T_{\iota}|
=
300\sum_H x_H
\le
\sum_{H_Q} 375\,m_{H_Q}
=
375|Q_{\iota}|
\le
625|Q_{\iota}|,
\]
which completes the proof. 
\end{proof}

\section{Proof of Theorem~\ref{thm:main}}\label{sec:r4-board-input-proof}
In this section, we prove Theorem~\ref{thm:main}. 
For a fixed $19$-board $V_\tau$, recall that $T_z$, $P_z$, and $Q_z$ are the sets of present triples, present pairs, and missing quads of spread $z$, respectively. Summarizing Theorems  \ref{THM:Lay1-main-result}, \ref{cor:layer2} and \ref{cor:layer3} gives 
\begin{align}
 \frac{12}{25}|T_3|+\frac{144}{625}|P_2|&\le |Q_4|,\label{eq:L1}\\
 \frac{24}{25}|T_2|+\frac{432}{625}|P_1|&\le |Q_3|,\label{eq:L2}\\
 \frac{36}{25}|T_1|&\le |Q_2|.\label{eq:L3}
\end{align}
Also $|T_0|=|P_0|=0$ by \eqref{z(H)}.
At the critical values $n\in\{n_4(s)-1,n_4(s)\}$ and for $s\ge3481$, Lemma~\ref{lem:r4-explicit-threshold} gives
$\frac q{s-3}\le\frac{12}{25}$.
Recall that $q=n-(4s+3)$ and $w(H)=\frac{\binom q{4-|H|}}{\binom{s-z(H)}{4-z(H)}}$.
Since this weight depends only on $|H|$ and $z(H)$, we can write $w(T_z)$
and $w(P_z)$ for the common weights of triples and pairs of spread $z$.
By Lemma~\ref{lem:r4-explicit-threshold}, we have 
\begin{align*}
 w(T_3)=\frac q{s-3}&\le\frac{12}{25},
 &
 w(P_2)=\frac{\binom q2}{\binom{s-2}2}&\le\frac{144}{625},\\
 (s-3)w(T_2)=\frac{2q}{s-2}&\le\frac{24}{25},
 &
 (s-3)w(P_1)=\frac{3q(q-1)}{(s-1)(s-2)}&\le\frac{432}{625},\\
 \binom{s-2}{2}w(T_1)=\frac{3q}{s-1}&\le\frac{36}{25}.
\end{align*}
The matching missing-quad weights are $1$, $1/(s-3)$, and $1/\binom{s-2}{2}$ for spreads $4$, $3$, and
$2$, respectively.

\begin{theorem}\label{thm:weighted-board}
Let \(s\ge 3481\) and \(n\in\{n_4(s)-1,n_4(s)\}\).
For every $\tau\in \binom{[s]}{4}$, the corresponding \(19\)-board \(V_\tau\) satisfies
\[
 \sum_{\substack{H\subseteq V_\tau, \ H\in \F_0,\ \\ |H|=2,3}} w(H)
 \le  \sum_{\substack{H_Q\subseteq V_\tau,\ H_Q\notin \F_0\\ |H_Q|=4}} w(H_Q).
\]
\end{theorem}
\begin{proof}
Because the weight is constant on each size-and-spread class, the contribution
of the spread-$3$ triples and spread-$2$ pairs is bounded by
\[
\begin{aligned}
        w(T_3)\cdot|T_3|+w(P_2)\cdot|P_2|\le \frac{12}{25}|T_3|+\frac{144}{625}|P_2|\le |Q_4|,
\end{aligned}
\]
where the last inequality is~\eqref{eq:L1}.  
A missing spread-$4$ quad has weight $1$, so the right-hand side is exactly the total weight of the missing quads used at this layer.
For the next layer, the displayed weight estimates and~\eqref{eq:L2} give
\[
\begin{aligned}
        w(T_2)\cdot |T_2|+w(P_1)\cdot |P_1|\le \frac{1}{s-3}
        \left(\frac{24}{25}|T_2|+\frac{432}{625}|P_1|\right)
        &\le \frac{1}{s-3}|Q_3|.
\end{aligned}
\]
This again matches the missing-quad weight, since every missing spread-$3$
quad has weight $1/(s-3)$.
Similarly, after dividing~\eqref{eq:L3} by $\binom{s-2}{2}$, we obtain
\[
\begin{aligned}
        w(T_1)\cdot |T_1|\le \frac{1}{\binom{s-2}{2}}
        \left(\frac{36}{25}|T_1|\right)
        &\le \frac{1}{\binom{s-2}{2}}|Q_2|,
\end{aligned}
\]
which is the total weight of the missing spread-$2$ quads.
Hence,
\[
 \sum_{\substack{H\subseteq V_\tau, \ H\in \F_0,\ \\ |H|=2,3}} w(H)
 \le  |Q_4|+\frac{1}{s-3}|Q_3|+\frac{1}{\binom{s-2}{2}}|Q_2|.
\]
The right-hand side is the total weight of all missing quads of spreads
$2,3,4$. 
No quad has spread $0$, and the missing quads of spread $1$ contribute only nonnegative terms to the right-hand side. This completes the proof. 
\end{proof}

Let $\rho=\rho_4$ be the root of $\rho^4-(\rho-1)^4=256$; numerically,
$\rho=4.479166\ldots<112/25$. We now complete the proof of Theorem~\ref{thm:main}.

\begin{proof}[Proof of Theorem~\ref{thm:main}]
For an $n$-vertex stable maximal property-$\ONE$ $4$-graph, Theorem~\ref{thm:weighted-board}
says that, on every $19$-board, the total weight of the present pairs and
triples is at most the total weight of the missing quads.  Since the comparison
family $\mathcal A_4(n,s)$ contains every quad trace and no smaller trace,
this is exactly the local comparison~\eqref{eq:local-board-general}.  Thus the
finite-board input holds for every matching number $s\ge3481$ and
$n\in\{n_4(s)-1,n_4(s)\}$.

Theorem~\ref{lem:FRR2017} verifies Assumption~\ref{ass:link-general} for
$r=4$ with $s_{\rm link}(4)=33$.  
Lemma~\ref{lem:r4-explicit-threshold} gives
$n_4(s)-n_4(s-1)\ge2$ for every $s\ge3481$, so the difference hypothesis in
the finite-board criterion holds with $s_\Delta(4)=3481$.

It remains to verify the quantitative lower estimate for $n_4(s)$.  Direct expansion of
$\Delta_s(x)=\binom{x}{4}-\binom{x-s}{4}-\binom{4s+3}{4}$ gives
\[
24\Delta_s\left(\rho s+\frac72\right)
=8(3\rho^2-3\rho-47)s^3
 +\left(43\rho-\frac{395}{2}\right)s^2-2s<0,
\]
where the last inequality follows from $\rho<112/25$.  Because $n_4(s)$
is the first integer at which $\Delta_s$ is nonnegative, it follows that
\[
        n_4(s)-1\ge \rho s+\frac52.
\]
We may therefore take $\eta_4=5/2$ in
Theorem~\ref{thm:finite-board-criterion}.  Its constants are
\[
        c_0=\frac{2r-2}{\rho-1}=\frac6{\rho-1},
        \qquad
        c_1=\frac{r-1-\eta_4}{\rho-1}
             =\frac{1/2}{\rho-1}.
\]
Applying Theorem~\ref{thm:finite-board-criterion} with $s_0=3481$ now yields
\[
        m_4(n,s)=\max\{a_4(s), b_4(n,s)\}
\]
for every $n\ge4s+4$ and every integer $s$ satisfying
\[
s\ge
\left\lceil\frac{6\cdot3481+1/2}{\rho-1}\right\rceil
=6004.
\]
This proves the theorem.
\end{proof}

\section*{Acknowledgments}
J.H. was supported by the National Key R\&D Program of China (No.~2023YFA1010202).
X.L. was supported by the Excellent Young Talents Program (Overseas) of the National Natural Science Foundation of China.

\section*{Declaration on the use of AI}
The authors used OpenAI GPT-5.5 solely to generate the code for the computer-assisted components of Section~\ref{sec:layer1}. 
The initial ideas underlying the remaining mathematical arguments and
proofs were inspired by the work of Frankl, R\"odl, and Ruci\'nski \cite{FRR2017}, and the arguments and proofs were subsequently developed independently by the authors without the use of AI-assisted tools. 
Each author has carefully verified all mathematical content in the paper and accepts
full responsibility for its correctness.

\bibliographystyle{abbrv}
\bibliography{matching}

@article{AlonEtAl2012,
  author  = {Alon, Noga and Frankl, Peter and Huang, Hao and R{\"o}dl, Vojt{\v e}ch and Ruci{\'n}ski, Andrzej and Sudakov, Benny},
  title   = {Large matchings in uniform hypergraphs and the conjectures of {Erd{\H{o}}s} and {Samuels}},
  journal = {J. Combin. Theory Ser. A},
  volume  = {119},
  year    = {2012},
  pages   = {1200--1215}
}

@article{BDE1976,
  author  = {Bollob{\'a}s, B{\'e}la and Daykin, David E. and Erd{\H{o}}s, Paul},
  title   = {Sets of independent edges of a hypergraph},
  journal = {Quart. J. Math. Oxford Ser. (2)},
  volume  = {27},
  year    = {1976},
  pages   = {25--32}
}

@misc{CaoLiuZhang2026,
  author = {Cao, Mengyu and Liu, Hong and Zhang, Haixiang},
  title  = {A Near-Optimal Linear Range for the {Erd{\H{o}}s} Matching Conjecture},
  year   = {2026},
  note   = {arXiv:2608.19118}
}

@article{EKR1961,
  author  = {Erd{\H{o}}s, Paul and Ko, Chao and Rado, Richard},
  title   = {Intersection theorems for systems of finite sets},
  journal = {Quart. J. Math. Oxford Ser. (2)},
  volume  = {12},
  year    = {1961},
  pages   = {313--320}
}

@article{Erdos1965,
  author  = {Erd{\H{o}}s, Paul},
  title   = {A problem on independent {$r$}-tuples},
  journal = {Ann. Univ. Sci. Budapest. E{\"o}tv{\"o}s Sect. Math.},
  volume  = {8},
  year    = {1965},
  pages   = {93--95}
}

@article{ErdosGallai1959,
  author  = {Erd{\H{o}}s, Paul and Gallai, Tibor},
  title   = {On maximal paths and circuits of graphs},
  journal = {Acta Math. Acad. Sci. Hungar.},
  volume  = {10},
  year    = {1959},
  pages   = {337--356}
}

@article{FLM2012,
  author  = {Frankl, Peter and {\L}uczak, Tomasz and Mieczkowska, Katarzyna},
  title   = {On matchings in hypergraphs},
  journal = {Electron. J. Combin.},
  volume  = {19},
  number  = {2},
  year    = {2012},
  pages   = {Paper 42}
}

@article{Frankl2013,
  author  = {Frankl, Peter},
  title   = {Improved bounds for {Erd{\H{o}}s}' matching conjecture},
  journal = {J. Combin. Theory Ser. A},
  volume  = {120},
  year    = {2013},
  pages   = {1068--1072}
}

@article{Frankl2017,
  author  = {Frankl, Peter},
  title   = {On the maximum number of edges in a hypergraph with given matching number},
  journal = {Discrete Appl. Math.},
  volume  = {216},
  year    = {2017},
  pages   = {562--581}
}

@article{Frankl2017NewRange,
  author  = {Frankl, Peter},
  title   = {Proof of the {Erd{\H{o}}s} matching conjecture in a new range},
  journal = {Israel J. Math.},
  volume  = {222},
  year    = {2017},
  pages   = {421--430}
}

@article{FranklKupavskii2017,
  author  = {Frankl, Peter and Kupavskii, Andrey},
  title   = {Families with no {$s$} pairwise disjoint sets},
  journal = {J. Lond. Math. Soc. (2)},
  volume  = {95},
  number  = {3},
  year    = {2017},
  pages   = {875--894}
}

@article{FranklKupavskii2022,
  author  = {Frankl, Peter and Kupavskii, Andrey},
  title   = {The {Erd{\H{o}}s} Matching Conjecture and concentration inequalities},
  journal = {J. Combin. Theory Ser. B},
  volume  = {157},
  year    = {2022},
  pages   = {366--400}
}

@misc{FranklLuMaWu2026,
  author = {Frankl, Peter and Lu, Hongliang and Ma, Jie and Wu, Yuze},
  title  = {Towards the {Erd{\H{o}}s} matching conjecture for 4-uniform hypergraphs: stability and applications},
  year   = {2026},
  note   = {arXiv:2602.19230}
}

@article{FRR2017,
  author  = {Frankl, Peter and R{\"o}dl, Vojt{\v e}ch and Ruci{\'n}ski, Andrzej},
  title   = {A short proof of {Erd{\H{o}}s}' conjecture for triple systems},
  journal = {Acta Math. Hungar.},
  volume  = {151},
  year    = {2017},
  pages   = {495--509}
}

@article{HuangZhao2017,
  author  = {Huang, Hao and Zhao, Yi},
  title   = {Degree versions of the {Erd{\H{o}}s--Ko--Rado} theorem and {Erd{\H{o}}s} hypergraph matching conjecture},
  journal = {J. Combin. Theory Ser. A},
  volume  = {150},
  year    = {2017},
  pages   = {233--247}
}

@article{Han2016,
  author  = {Han, Jie},
  title   = {Perfect matchings in hypergraphs and the {Erd{\H{o}}s} matching conjecture},
  journal = {SIAM J. Discrete Math.},
  volume  = {30},
  number  = {3},
  year    = {2016},
  pages   = {1351--1357}
}

@article{LuczakM2014,
  author  = {{\L}uczak, Tomasz and Mieczkowska, Katarzyna},
  title   = {On {Erd{\H{o}}s}' extremal problem on matchings in hypergraphs},
  journal = {J. Combin. Theory Ser. A},
  volume  = {124},
  year    = {2014},
  pages   = {178--194}
}

@article {HLS2012,
    AUTHOR = {Huang, Hao and Loh, Po-Shen and Sudakov, Benny},
     TITLE = {The size of a hypergraph and its matching number},
   JOURNAL = {Combin. Probab. Comput.},
  FJOURNAL = {Combinatorics, Probability and Computing},
    VOLUME = {21},
      YEAR = {2012},
    NUMBER = {3},
     PAGES = {442--450},
      ISSN = {0963-5483,1469-2163},
   MRCLASS = {05D15 (05C65)},
  MRNUMBER = {2912790},
MRREVIEWER = {David\ J.\ Galvin},
       DOI = {10.1017/S096354831100068X},
       URL = {https://doi.org/10.1017/S096354831100068X},
}

@article {K1968,
    AUTHOR = {Kleitman, Daniel J.},
     TITLE = {Maximal number of subsets of a finite set no {$k$} of which
              are pairwise disjoint},
   JOURNAL = {J. Combinatorial Theory},
  FJOURNAL = {Journal of Combinatorial Theory},
    VOLUME = {5},
      YEAR = {1968},
     PAGES = {157--163},
      ISSN = {0021-9800},
   MRCLASS = {05.04},
  MRNUMBER = {228350},
MRREVIEWER = {H.\ V.\ Kronk},
}

@article {FRR2012,
    AUTHOR = {Frankl, Peter and R\"odl, Vojtech and Ruci\'nski, Andrzej},
     TITLE = {On the maximum number of edges in a triple system not
              containing a disjoint family of a given size},
   JOURNAL = {Combin. Probab. Comput.},
  FJOURNAL = {Combinatorics, Probability and Computing},
    VOLUME = {21},
      YEAR = {2012},
    NUMBER = {1-2},
     PAGES = {141--148},
      ISSN = {0963-5483,1469-2163},
   MRCLASS = {05D05 (05C65 05C70)},
  MRNUMBER = {2900053},
MRREVIEWER = {Zbigniew\ B.\ Lonc},
       DOI = {10.1017/S0963548311000496},
       URL = {https://doi.org/10.1017/S0963548311000496},
}

@article {GLM2022,
    AUTHOR = {Guo, Mingyang and Lu, Hongliang and Mao, Dingjia},
     TITLE = {A stability result on matchings in 3-uniform hypergraphs},
   JOURNAL = {SIAM J. Discrete Math.},
  FJOURNAL = {SIAM Journal on Discrete Mathematics},
    VOLUME = {36},
      YEAR = {2022},
    NUMBER = {3},
     PAGES = {2339--2351},
      ISSN = {0895-4801,1095-7146},
   MRCLASS = {05C70 (05C65 05C72)},
  MRNUMBER = {4487902},
MRREVIEWER = {Ioan\ Tomescu},
       DOI = {10.1137/21M1422720},
       URL = {https://doi.org/10.1137/21M1422720},
}

@article {LWY2022,
    AUTHOR = {Lu, Hongliang and Wang, Yan and Yu, Xingxing},
     TITLE = {Rainbow perfect matchings for 4-uniform hypergraphs},
   JOURNAL = {SIAM J. Discrete Math.},
  FJOURNAL = {SIAM Journal on Discrete Mathematics},
    VOLUME = {36},
      YEAR = {2022},
    NUMBER = {3},
     PAGES = {1645--1662},
      ISSN = {0895-4801,1095-7146},
   MRCLASS = {05C65 (05C35 05C70)},
  MRNUMBER = {4451911},
MRREVIEWER = {Anna\ A.\ Taranenko},
       DOI = {10.1137/21M1442383},
       URL = {https://doi.org/10.1137/21M1442383},
}

@article {LYY2021,
    AUTHOR = {Lu, Hongliang and Yu, Xingxing and Yuan, Xiaofan},
     TITLE = {Rainbow matchings for 3-uniform hypergraphs},
   JOURNAL = {J. Combin. Theory Ser. A},
  FJOURNAL = {Journal of Combinatorial Theory. Series A},
    VOLUME = {183},
      YEAR = {2021},
     PAGES = {Paper No. 105489, 21},
      ISSN = {0097-3165,1096-0899},
   MRCLASS = {05C70 (05C15 05C65)},
  MRNUMBER = {4275007},
MRREVIEWER = {Mark\ Rowland\ Budden},
       DOI = {10.1016/j.jcta.2021.105489},
       URL = {https://doi.org/10.1016/j.jcta.2021.105489},
}

@article {LWY2023,
    AUTHOR = {Lu, Hongliang and Wang, Yan and Yu, Xingxing},
     TITLE = {Co-degree threshold for rainbow perfect matchings in uniform
              hypergraphs},
   JOURNAL = {J. Combin. Theory Ser. B},
  FJOURNAL = {Journal of Combinatorial Theory. Series B},
    VOLUME = {163},
      YEAR = {2023},
     PAGES = {83--111},
      ISSN = {0095-8956,1096-0902},
   MRCLASS = {05C70 (05C07 05C65)},
  MRNUMBER = {4630904},
MRREVIEWER = {Vahan\ V.\ Mkrtchyan},
       DOI = {10.1016/j.jctb.2023.07.003},
       URL = {https://doi.org/10.1016/j.jctb.2023.07.003},
}

@article {GLMY2022,
    AUTHOR = {Gao, Jun and Lu, Hongliang and Ma, Jie and Yu, Xingxing},
     TITLE = {On the rainbow matching conjecture for 3-uniform hypergraphs},
   JOURNAL = {Sci. China Math.},
  FJOURNAL = {Science China. Mathematics},
    VOLUME = {65},
      YEAR = {2022},
    NUMBER = {11},
     PAGES = {2423--2440},
      ISSN = {1674-7283,1869-1862},
   MRCLASS = {05C65 (05D05)},
  MRNUMBER = {4499524},
MRREVIEWER = {Paola\ Bonacini},
       DOI = {10.1007/s11425-020-1890-4},
       URL = {https://doi.org/10.1007/s11425-020-1890-4},
}

\newpage

\appendix

\section{Machine-checkable certificate specifications}\label{app:certificate-specs}
This appendix describes the two computer-assisted components of the proof.  Section~ \ref{app:topstar-spec} reconstructs the top-star linear relaxation and checks an exact integer-scaled Farkas certificate.
Section~\ref{app:notopstar-spec} uses deterministic exhaustive searches to certify the exact minima or target-bounded lower bounds required in the no-top-star branch.  
The accompanying JSON files contain the corresponding certificate data or search specifications, and all proof-critical checks use exact arithmetic.

\subsection{Top-star Farkas system}\label{app:topstar-spec}
The following row types are the mathematical specification of the top-star branch relaxation.  Each row type is justified below as a necessary condition for a genuine top-star board. Put
\[
        Z:=\{x\in[4]^3:\text{at least one coordinate of }x\text{ is }1\}.
\]
\noindent\textbf{(R1)}~ For \(x\in [4]^3\setminus Z\), $y_x\le 0$; for \(x\in [4]^3\) with \(x_i>1\),
$y_x\le y_{x-e_i}$, where $e_i$ denotes the $i$-th standard basis vector. 

\textbf{Note:}~The first inequality follows from \(S^+_4\subseteq Q_4\), and the second follows from stability.

\medskip
\noindent\textbf{(R2)}~If $(x,t)\le (x',t')$ coordinatewise with $t,~t'\in\{1,2,3\}$, then $m_{x,t}\le m_{x',t'}$.

If $(x,t)\le (x',4)$ coordinatewise with $t\in\{1,2,3\}$, then $m_{x,t}+y_{x'}\le 1$.

\textbf{Note:}~These rows express the up-set property of $Q_4$, which is
equivalent to stability of the present wide-quad traces.

\medskip
\noindent\textbf{(R3)}~$\ell\le c+2$ is written in variables as
\[
        \sum_{x\in [4]^3}\sum_{t=1}^3 m_{x,t}
        -
        \sum_{x\in [4]^3} y_x
        \le 2.
\]

\medskip

\noindent\textbf{(R4)}~For every triple support \(I\), if \(\mathbf{v}, \mathbf{v}'\in [4]^3\) and \(\mathbf{v}'\le \mathbf{v}\)
coordinatewise, then
\[
        t_{I,\mathbf{v}}\le t_{I\mathbf{,}\mathbf{v}'}.
\]
For every pair support \(J\), if \(\mathbf{u}, \mathbf{u}'\in[4]^2\) and \(\mathbf{u}'\le \mathbf{u}\) coordinatewise, then
\[
        p_{J, \mathbf{u}}\le p_{J,\mathbf{u}'}.
\]
\textbf{Note:}~These rows again follow from stability.

\medskip

\noindent\textbf{(R5)}~Let \(H\) be a supported pair or triple trace, and let \(z_H\) be its indicator,
and let \(Q\supset H\) be a wide quad extending the corresponding trace.

If \(Q=(x,4)\), then
\[
        z_H\le y_x.
\]
If \(Q=(x,t)\), \(t\in\{1,2,3\}\), then
\[
        z_H+m_{x,t}\le 1.
\]
\textbf{Note:}~By Lemma~\ref{fact:closure-general}, every wide quad containing a present trace of size two or three is present.

\medskip

\noindent\textbf{(R6)}~For every pair support \(J\), and for every two levels
$\mathbf{a}=(a_1,a_2)$, $\mathbf{b}=(b_1,b_2)\in[4]^2$ satisfying $a_1\ne b_1$, $a_2\ne b_2$, and $\{a_r,b_r\}\ne\{1,2\}$ for some $r\in\{1,2\}$, we impose
\[
        p_{J, \mathbf{a}}+p_{J, \mathbf{b}}\le 1.
\]
\textbf{Note:}~If both pair traces were present, then seed $R_{\F}\cup \{e_{i,{1~\text{or}~2}}\}$ where $i$ is a column related to $a_r$, together with the two full columns outside \(J\), would give five pairwise disjoint traces, contradicting
Lemma~\ref{lem:local-matching-r4}.

\medskip

\noindent\textbf{(R7)}~For every triple support \(I\), and for every three levels
\(\mathbf{a}, \mathbf{b}, \mathbf{c}\in [4]^3\) such that in every coordinate the three used levels are pairwise distinct and in at least one coordinate the unused level is \(1\) or \(2\), we impose
\[
        t_{I,\mathbf{a}}+t_{I,\mathbf{b}}+t_{I,\mathbf{c}}\le 2.
\]
\textbf{Note:}~If all three triples were present, then the corresponding seed and the fourth full column would give five pairwise disjoint traces, contradicting Lemma~\ref{lem:local-matching-r4}.

\medskip
\noindent\textbf{(R8)}~Let \(H\) be a supported pair or triple trace, with indicator \(z_H\).  Choose a coordinate \(j\in\{1,2,3,4\}\) and a level \(u\in\{1,2\}\) not used by \(H\) in the coordinate \(j\). The seed $R_{\F}\cup\{e_{j,u}\}$ is present by Lemma~\ref{lem:seeds-r4}.  

Let \(R=\{Q_1, Q_2, Q_3\}\) be three pairwise disjoint wide quads, all disjoint from both \(H\) and this seed. Write $Q_i=(x_i, t_i)\in [4]^3\times\{1,2,3,4\}$, where \(t_i\) is the fourth coordinate, i.e. the top-star direction. Put
\[
R_{\rm top}:=\{i:t_i=4\},\qquad
R_{\rm low}:=\{i:t_i\le 3\}.
\]
For \(i\in R_{\rm top}\), the variable \(y_{x_i}\) records whether the top-slice quad \(Q_i=(x_i,4)\) is present. For \(i\in R_{\rm low}\), the variable \(m_{x_i,t_i}\) records whether the lower-slice quad
\(Q_i=(x_i,t_i)\) is missing.
We impose
\[
z_H+\sum_{i\in R_{\rm top}} y_{x_i}
   -\sum_{i\in R_{\rm low}} m_{x_i,t_i}
\le |R_{\rm top}|.
\]
Equivalently, $z_H\le \sum_{i\in R_{\rm top}}(1-y_{x_i})+\sum_{i\in R_{\rm low}}m_{x_i,t_i}$.

(For rows of type \textbf{(R8)}, the verifier uses the fixed cyclic subfamily generated by \texttt{residual\_three\_quads}: after ordering the three residual quads by their first coordinates, it uses the three cyclic alignments in each of the remaining
coordinates.  
Every generated row is a valid instance of {\rm (R8)}.  Since this gives a weaker relaxation, proving the objective bound for this subfamily is sufficient for every
genuine board.)

\textbf{Note:}~The case of $z_H=0$ is obvious. For \(z_H=1\), this inequality is equivalent to saying that at least one of the three remaining quads \(Q_1,Q_2,Q_3\) is missing:
\[
        \sum_{(x,4)\in R_{\rm top}}(1-y_x)
        +
        \sum_{(x,t)\in R_{\rm low}}m_{x,t}
        \ge 1.
\]
If this failed, then \(H\), the seed, and all three quads \(Q_1,Q_2,Q_3\) would be pairwise disjoint present traces in the same \(19\)-board, contradicting Lemma~\ref{lem:local-matching-r4}.

\medskip

\noindent\textbf{(R9)}~For every pair support \(J\),
\[
        \sum_{\mathbf{u}\in[4]^2} p_{J, \mathbf{u}}\le 4,
\]
and for every triple support \(I\),
\[
        \sum_{\mathbf{v}\in[^3} t_{I, \mathbf{v}}\le 32.
\]
\textbf{Note:}~The first bound is Lemma~\ref{lem:pair-ferrers}; the second is Lemma~\ref{lem:triple-ferrers}.

\medskip

\noindent\textbf{(R10)}
The branch set is
\[
        \mathcal B
        :=
        \{c:0\le c\le 37,\ c\ne 23\}
        \cup
        \{(23,\ell):0\le \ell\le 25\}.
\]
For every branch, \(c\) is fixed by
\[
\sum_{x\in[4]^3}y_x\le c,
\qquad
-\sum_{x\in[4]^3}y_x\le -c.
\]

For the down-set \(C^{\rm top}\subseteq V_3\) of size \(c\), 
\[
 \text{if}~~|\downarrow_Z x=\{z\in Z:z\preceq x\}|>c,~~\text{then we add}~~y_x\le0.
\]
and 
\[
 \text{if}~~|\uparrow_Z x=\{z\in Z:z\succeq x\}|>|Z|-c=37-c,~~\text{then we add}~~-y_x\le -1.
\]

In the exceptional branch \(\beta=(23,\ell)\), the lower missing size is
also fixed by
\[
\sum_{x\in[4]^3}\sum_{t=1}^3 m_{x,t}\le\ell,
\qquad
-\sum_{x\in[4]^3}\sum_{t=1}^3 m_{x,t}\le-\ell.
\]

Let $L:=[4]^3\times\{1,2,3\}$. The lower missing family \(N\subseteq L\) is an up-set of size \(\ell\).
For \(q\in L\),
\[
 \text{if}~~|\uparrow_L q=\{r\in L:q\preceq r\}|>\ell,~~\text{then we add}~~m_q\le0.
\]
and 
\[
 \text{if}~~|\downarrow_L q:=\{r\in L: q\succeq r\}|>|L|-\ell=192-\ell,~~\text{then we add}~~-m_q\le -1.
\]
\textbf{Note:}~A down-set of size \(c\) cannot contain a point whose principal down-set has size larger than \(c\), and it must contain a point whose principal up-set is too large to be avoided by a set of size \(c\).
The verifier 
\[
\texttt{src/topstar.py}
\]
reconstructs the finite linear relaxation used by the certificate from the row rules \textup{(R1)}--\textup{(R10)}. The variable vector is
\[
X=(y,m,t,p),
\]
where \(y\) records present top-slice quads, \(m\) records missing lower quads, \(t\) records present supported triples, and \(p\) records present supported pairs. The objective vector \(d\) has coefficients
\[
625 \text{ on } y,\qquad -625 \text{ on } m,\qquad
300 \text{ on } t,\qquad 144 \text{ on } p.
\]
Thus
\[
d^T X
= 625\sum_x y_x -625\sum_{x,t}m_{x,t}+300\sum_{I,\mathbf{v}}t_{I,\mathbf{v}}+144\sum_{J,\mathbf{u}}p_{J,\mathbf{u}}.
\]

For every branch
\[
\beta\in
\{c:0\le c\le 37,\ c\ne 23\}
\cup
\{(23,\ell):0\le \ell\le 25\},
\]
the verifier regenerates the matrix inequality
\[
A_\beta X\le b_\beta
\]
from the mathematical definitions. The verifier does not read \(A_\beta\), \(b_\beta\), or \(d\) from the certificate file. Instead, it reconstructs them from the row rules \textup{(R1)}--\textup{(R10)}. The certificate file supplies only the Farkas data \(D_\beta,\lambda_\beta,\mu_\beta\).
The file 
\[
\texttt{certificates/topstar\_farkas.json}
\]
contains, for each branch \(\beta\), a positive integer \(D_\beta\), nonnegative row
multipliers \(\lambda_\beta\), and nonnegative upper-bound multipliers
\(\mu_\beta\). The verifier checks
\[
A_\beta^T\lambda_\beta+\mu_\beta\ge D_\beta d
\tag{F1}
\]
coordinatewise and
\[
b_\beta^T\lambda_\beta+\mathbf 1^T\mu_\beta\le 40000D_\beta .
\tag{F2}
\]
Equivalently, it computes
\[
g_\beta=A_\beta^T\lambda_\beta+\mu_\beta-D_\beta d,
\qquad
h_\beta=40000D_\beta-b_\beta^T\lambda_\beta-\mathbf 1^T\mu_\beta
\]
and checks that all coordinates of \(g_\beta\) and all \(h_\beta\) are
nonnegative integers.

The submitted run reports
\[
\texttt{topstar\_certificate\_ok=True},
\quad
\texttt{branches=63},
\quad
\texttt{min\_gap=0},
\quad
\texttt{min\_rhs\_gap=386598099}.
\]
Hence \textup{(F1)} and \textup{(F2)} hold for every branch, and the usual
Farkas-duality argument gives \(d^T X\le 40000\), which is precisely
\[
300|T_3|+144|P_2|\le 625|Q_4|.
\]

\subsection{No-top-star finite universe}\label{app:notopstar-spec}

The no-top-star branch uses the universe $V_4\coloneqq [4]^4$ and the upper
cube $U\coloneqq\{2,3,4\}^4$.  
The missing wide-quad set is an up-set $Q_4\subseteq V_4$, and the present upper cube is $C\coloneqq U\setminus Q_4$.
The finite theorem uses pair trace sets $P_J\subseteq [4]^2$ for $J\in\binom{[4]}2$ and triple trace sets $T_I\subseteq [4]^3$ for
$I\in\binom{[4]}3$.

The finite no-top-star system consists of the following exact conditions.
\begin{enumerate}[label=(N\arabic*)]
\item $Q_4$ is an up-set in $V_4$.
\item $C$ contains no upper $3$-matching.  Equivalently, for every choice of
permutations $\pi_1,\pi_2,\pi_3$ of $\{2,3,4\}$, at least one of the three
points $(t,\pi_1(t),\pi_2(t),\pi_3(t))$, $t=2,3,4$, lies in $Q_4$.
\item $Q_4\cap U$ contains no full top-star, equivalently none of
$(4,2,2,2)$, $(2,4,2,2)$, $(2,2,4,2)$, $(2,2,2,4)$ lies in $Q_4$.
\item Each $P_J$ is a down-set avoiding the forbidden pair obstruction
from (R6), and each $T_I$ is a down-set avoiding the bad triple
configuration from (R7).
\item Closure and remaining seed hitting hold.  If a trace $H$ is present, then
every wide quad extension of $H$ lies outside $Q_4$.  For every seed disjoint from $H$, every remaining 3-matching $R=(H_{Q_1}, H_{Q_2}, H_{Q_3})$ disjoint from $H$ and the seed must be hit by $Q_4$, i.e.$\{H_{Q_1}, H_{Q_2}, H_{Q_3}\}\cap Q_4 \neq \emptyset$.
\end{enumerate}

We briefly recall why every genuine no-top-star board satisfies
(N1)--(N5).  Condition (N1) is exactly stability of the present wide quads.
Condition (N2) is the assumption that the present upper cube has no upper
3-matching.  Condition (N3) is the no-top-star assumption.  
The conditions in (N4) again follow from stability, and the forbidden pair and
bad triple obstructions follow from the local five-trace obstruction:
otherwise the offending pair or triple traces, together with a seed from
Lemma~\ref{lem:seeds-r4} and the unused full columns, would give five pairwise disjoint present traces in the same $19$-board, contradicting Lemma~\ref{lem:local-matching-r4}.  
Finally, completion closure in (N5) is Lemma~\ref{fact:closure-general}, and remaining seed hitting follows because if \(H,S,H_{Q_1}, H_{Q_2}, H_{Q_3}\) were all present and pairwise disjoint, they would again form five pairwise disjoint present traces, contradicting Lemma~\ref{lem:local-matching-r4}.

The finite conditions are exactly \((N1)\)--\((N5)\).
The remaining seed matching generator is used to encode the local
five-trace obstruction. Fix a supported trace \(H\) and a seed \(S\) disjoint from \(H\). In coordinate \(j\), let $A_j\subseteq\{1,2,3,4\}$ be the set of levels already occupied by \(H\) and \(S\) in that coordinate.

A remaining triple \(R=(H_{Q_1}, H_{Q_2}, H_{Q_3})\) of wide quads is disjoint from both \(H\) and \(S\), and is pairwise disjoint, if and only if for every coordinate \(j\), the three levels
\[
        H_{Q_1}(j),~H_{Q_2}(j),~H_{Q_3}(j)   
\]
are three distinct elements of $\{1,2,3,4\}\setminus A_j$. The verifier generates these triples coordinate by coordinate
and imposes that each such triple is hit by $Q_4$.
We next record the conclusions supplied by the finite verifiers.

\begin{conclusion}\label{Con:no-topstar-upper-blocker}
   Every finite configuration satisfying (N1)--(N3) has
\[
        |Q_4\cap U|\ge 33.
\]
\end{conclusion}

\begin{proof}[Certificate description]
This is checked by the source file
\[
        \texttt{src/no\_topstar\_upper\_blocker.cpp}.
\]

By \(U=\{2,3,4\}^4\), (N3) becomes the four forced inclusions, which we call these the forced unit-top conditions,
\[
        (4,2,2,2),\quad
        (2,4,2,2),\quad
        (2,2,4,2),\quad
        (2,2,2,4)
        \in C.
\]
The verifier enumerates down-sets $ C\subseteq\{2,3,4\}^4$ by \(3\times3\times3\) height matrices. Namely, for each $(a,b,c)\in\{2,3,4\}^3$, because \(C\) is a down-set,
\[
        (a,b,c,d)\in C
        \quad\Longleftrightarrow\quad
        2\le d<2+h[a,b,c]~\text{where}~h[a,b,c]\in\{0,1,2,3\}.
\]
The verifier first checks an explicit down-set \(C\) of size \(48\) satisfying the forced unit-top conditions and containing no upper \(3\)-matching.
It then exhaustively enumerates all monotone \(3\times3\times3\) height matrices representing down-sets \(C\subseteq\{2,3,4\}^4\) of size at least \(49\) satisfying the forced unit-top inclusions, and verifies that each such candidate contains an upper \(3\)-matching. Hence no admissible
candidate has size at least \(49\), so \(|C|\le48\) (hence $|Q_4\cap U|=|U|-|C|\ge 81-48=33$).
The successful run reports
\[
        \texttt{max\_C=48},\qquad
        \texttt{min\_Q4U=33},\qquad
        \texttt{exact\_upper\_blocker\_branch\_bound\_ok=True}.
\]
\end{proof}

When $|Q_4|=m$, let \(p(m)\) be an upper bound for the number of present pair traces on one fixed pair support, 
and let \(t(m)\) be an upper bound for the number of present triple traces on one fixed triple support.
\begin{conclusion}\label{lem:no-topstar-trace-thresholds}
    Put \(m:=|Q_4|\). If $33\le m\le 63$, then every pair support \(J\) and every triple support \(I\) satisfy
\[
        |P_J|\le p(m),
        \qquad
        |T_I|\le t(m),
\]
where \(p(m)\) and \(t(m)\) are given in Table~\ref{tab:no-topstar-thresholds}.
\begin{table}[h]
\[
\begin{array}{c|cccccccc}
m
&33&34\text{--}35&36\text{--}41&42\text{--}47
&48\text{--}49&50\text{--}55&56\text{--}59&60\text{--}63\\
\hline
p(m)
&0&0&0&0&0&0&0&1\\
t(m)
&0&1&4&7&25&26&29&29
\end{array}
\]
\caption{Per-support trace bounds in the no-top-star threshold range.}
\label{tab:no-topstar-thresholds}
\end{table}
\end{conclusion}
\begin{proof}[Certificate description]
The certificate file
\[
        \texttt{certificates/no\_topstar\_threshold.json}
\]
contains the threshold records for the \(10\) pair-trace orbits and the
\(20\) triple-trace orbits, where the orbit means modulo permutations of the selected coordinates: the pair-trace representatives are the unordered pairs \(1\le a\le b\le 4\), giving \(\binom{4+2-1}{2}=10\), and the triple-trace representatives are the unordered triples
\(1\le a\le b\le c\le 4\), giving \(\binom{4+3-1}{3}=20\).


For each supported trace orbit \(H\), the C++ solver
\[
        \texttt{src/no\_topstar\_search.cpp}.
\]
certifies an integer threshold \(\mu(H)\) such that every up-set \(Q'_4\subseteq[4]^4\) satisfying \textnormal{(N1)}--\textnormal{(N3)}, the completion closure for \(H\), and all remaining seed-hitting constraints for \(H\) satisfies
\[
|Q'_4|\ge \mu(H).
\]
For records marked \texttt{exact}, the value \(\mu(H)\) is the true minimum; for the remaining records, only the stated lower bound is asserted.  
Consequently, if \(H\) is present in a feasible configuration with \(|Q_4|=m\), then necessarily \(\mu(H)\le m\).  
The Python verifier
\[
        \texttt{src/no\_topstar.py}
\]
combines this necessary condition with the legal down-set conditions in (N4), yielding
the per-support bounds \(p(m)\) and \(t(m)\) in Table~\ref{tab:no-topstar-thresholds}.

This use of \(\mu(H)\) is one-sided and therefore safe: the verifier only
uses individual trace feasibility as a necessary condition.  Allowing all
traces that are individually feasible can only enlarge the set of
admissible trace families, so the resulting levels $( p(m),t(m))$ are valid
upper bounds for genuine no-top-star boards.
The verifier checks \(30\) threshold cases, corresponding to the
\(10+20\) orbit representatives.  The low-range arithmetic is then the
direct integer check
\[
        300\cdot 4t(m)+144\cdot 6p(m)\le 625m,
        \qquad 33\le m\le 63.
\]
\end{proof}

\begin{conclusion}\label{lem:no-topstar-critical-range}
Put \(m:=|Q_4|\). Suppose that $64\le m\le 66$.
If some support is extremal, meaning that either
\[
        |P_J|=4
        \quad\text{for some }J\in\binom{[4]}2,
\]
or
\[
        |T_I|=32
        \quad\text{for some }I\in\binom{[4]}3,
\]
then
\[
        300|T_3|+144|P_2|-625|Q_4|\le -10624.
\]
\end{conclusion}
\begin{proof}[Certificate description]
The critical-range certificate treats the range \(64\le |Q_4|\le 66\).
The critical-pattern part of
\[        \texttt{certificates/no\_topstar\_threshold.json}
\]
contains nine exact pattern records, all checked by
\[
        \texttt{src/no\_topstar\_search.cpp}.
\]
Up to permutation of the four coordinates, these records are the following

On a pair support \((i,j)\), the threshold certificate gives $\mu(1,3), \mu(1,4)\ge 67$, thus the only possible extremal pair pattern is
\[
        \{1,2\}\times\{1,2\}.
\]
On a triple support \((i,j,k)\), with omitted coordinate \(\ell\), the
extremal triple slabs are
\[
        \{(a,b,c)\in[4]^3:a\le 2\},
        \qquad
        \{(a,b,c)\in[4]^3:b\le 2\},
        \qquad
        \{(a,b,c)\in[4]^3:c\le 2\}.
\]
The mixed triple equality pattern is the down-set $\{(a,b,c): 1\le c \le h_{ab}\}$, with height matrix
\[
        h=(h_{ab})_{a,b=1}^{4}
        \begin{pmatrix}
        4&4&2&2\\
        4&4&2&2\\
        2&2&0&0\\
        2&2&0&0
        \end{pmatrix},
\]
which comes from the equality case in the proof of Lemma~\ref{lem:triple-ferrers}: after decomposing \([4]^3\) by \(S=\{1,2\}\) and \(L=\{3,4\}\), the non-slab equality model is
\(S^3\cup(L\times S\times S)\cup(S\times L\times S)\cup(S\times S\times L)\), whose size is \(32\).
Here, for two distinct coordinates \(a,b\in[4]\), we write
\[
R_{ab}:=\{q=(q_1,q_2,q_3,q_4)\in\{1,2,3,4\}^4:
q_a\ge3,\ q_b\ge3\}.
\]
Thus \(R_{ab}\) is the high-level rectangle of wide quads whose \(a\)-th and
\(b\)-th coordinates both lie in the upper two levels \(\{3,4\}\). 
For the pair square, the solver proves minimum \(|Q_4|=64\), and under
\(|Q_4|\le66\) it forces $     R_{ij}\subseteq Q_4$.
For the three triple slabs, it proves minimum \(|Q_4|=64\), and under
\(|Q_4|\le66\) it respectively forces
\[
        R_{i\ell}\subseteq Q_4,
        \qquad
        R_{j\ell}\subseteq Q_4,
        \qquad
        R_{k\ell}\subseteq Q_4.
\]
For the mixed triple equality pattern, the current certificate record is
\[
        \texttt{triple\_mixed\_min80},
        \qquad
        \texttt{--pattern triple\_mixed --target 81},
        \qquad
        \texttt{expect: best=80}.
\]
Thus every feasible configuration realizing this pattern has \(|Q_4|\ge 80\).  Hence this pattern
cannot occur in the critical range \(|Q_4|\le66\).
The rectangle forcing is certified by contradiction.  
The negation \(R_{ab}\nsubseteq Q_4\) is equivalent to declaring this minimal point (\(q_a=q_b=3\) and all other coordinates \(1\)) to lie outside \(Q_4\).  The solver adds this negation to the finite system and proves the target-bounded search finds no feasible solution with \(|Q_4|\le66\).
Hence the rectangle containment \(R_{ab}\subseteq Q_4\) is forced.

After a rectangle \(R_{ij}\subseteq Q_4\) is forced, the Python
rectangle-extension audit enumerates all up-set extensions of the six
rectangles \(R_{ij}\) of total size at most \(66\). There are \(72\)
cases. For each cases, the verifier recomputes the maximum legal trace contribution using \((N4)\)--\((N5)\). The worst case has
\[
        |Q_4|=64,\qquad |T_3|=96,\qquad |P_2|=4,
\]
and gives
\[
        300\cdot 96+144\cdot 4-625\cdot 64=-10624<0.
\]
Therefore every extremal-support configuration in the range
\(64\le |Q_4|\le66\) satisfies the claimed bound.
The submitted verifier run reports
\[
\begin{gathered}
\texttt{no\_topstar\_certificate\_ok=True},\\
\texttt{trace\_threshold\_checks=30},\qquad
\texttt{critical\_pattern\_checks=9},\\
\texttt{rectangle\_cases=72},\qquad
\texttt{rectangle\_max\_case=(64,96,4,-10624)}.
\end{gathered}
\]
\end{proof}

Together with the top-star verifier, the top-level command
\[
        \texttt{python3 run\_all.py}
\]
finishes with
\[
        \texttt{topstar\_no\_topstar\_certificate\_ok=True}.
\]

\end{document}